\documentclass[reqno]{amsart}

\usepackage{mathrsfs}
\usepackage{amscd}
\usepackage{mathtools}
\usepackage{xcolor}
\usepackage{enumerate}

\theoremstyle{plain}
\newtheorem{theorem}{Theorem}[section]
\theoremstyle{remark}
\newtheorem{remark}[theorem]{Remark}

\theoremstyle{plain}

\newtheorem{lemma}[theorem]{Lemma}

\newtheorem{definition}[theorem]{Definition}

\newtheorem{assumption}[theorem]{Assumption}

\def\avint_#1{\mathchoice%
	{\mathop{\kern 0.2em\vrule width 0.6em height 0.69678ex depth -0.58065ex
			\kern -0.8em \intop}\nolimits_{\kern -0.4em#1}}%
	{\mathop{\kern 0.1em\vrule width 0.5em height 0.69678ex depth -0.60387ex
			\kern -0.6em \intop}\nolimits_{#1}}%
	{\mathop{\kern 0.1em\vrule width 0.5em height 0.69678ex depth -0.60387ex
			\kern -0.6em \intop}\nolimits_{#1}}%
	{\mathop{\kern 0.1em\vrule width 0.5em height 0.69678ex depth -0.60387ex
			\kern -0.6em \intop}\nolimits_{#1}}}

\newcommand{\N}{\ensuremath{\mathbb{N}}}
\newcommand{\Z}{\ensuremath{\mathbb{Z}}}

\newcommand{\R}{\ensuremath{\mathbb{R}}}
\newcommand{\bR}{\ensuremath{\mathbb{R}}}
\newcommand{\C}{\ensuremath{\mathbb{C}}}

\newcommand{\B}{\ensuremath{\mathcal{B}}}

\newcommand{\tr}{\text{tr}}
\def\cX{\mathcal{X}}

\newcommand{\osc}{{\rm osc}}

\DeclareMathOperator{\esssup}{ess.\,sup}

\def\Xint#1{\mathchoice
	{\XXint\displaystyle\textstyle{#1}}%
	{\XXint\textstyle\scriptstyle{#1}}%
	{\XXint\scriptstyle\scriptscriptstyle{#1}}%
	{\XXint\scriptscriptstyle%
		\scriptscriptstyle{#1}}%
	\!\int}
\def\XXint#1#2#3{{\setbox0=\hbox{$#1{#2#3}{%
				\int}$ }
		\vcenter{\hbox{$#2#3$ }}\kern-.6\wd0}}

\def\dashint{\Xint-}

\numberwithin{equation}{section}

\begin{document}
\numberwithin{equation}{section}
	
\author{Hongjie Dong}
\address{Division of Applied Mathematics\\
Brown University\\ Providence RI 02912, USA} \email{hongjie\_dong@bown.edu}
	
\author{Chiara Gallarati}
\address{Delft Institute of Applied Mathematics\\
Delft University of Technology \\ P.O. Box 5031\\ 2600 GA Delft\\The
Netherlands} \email{C.Gallarati@tudelft.nl}
	
\date\today
	
\title[Higher-order parabolic equations]
{Higher-order parabolic equations with VMO assumptions and general boundary conditions with variable leading coefficients}
	
\begin{abstract}
We prove weighted mixed $L_{p}(L_{q})$-estimates, with $p,q\in(1,\infty)$, for higher-order elliptic and parabolic equations on the half space $\R^{d+1}_{+}$ and on domains with general boundary conditions which satisfy the Lopatinskii--Shapiro condition.
We assume that the elliptic operators $A$ have leading coefficients which are in the class of vanishing mean oscillations both in the time and the space variables, and that the boundary conditions have variable leading coefficients. The proofs are based on and generalize the estimates recently obtained by the authors in \cite{DG17}.

\end{abstract}
	
\thanks{H. Dong was partially supported by the NSF under agreements DMS-1056737 and DMS-1600593. \\
\indent C. Gallarati was supported by the Vrije Competitie subsidy 613.001.206 of the Netherlands Organisation for Scientific Research (NWO)}

\maketitle

\section{Introduction}
	
In this paper we study the higher-order parabolic equation
\begin{equation}\label{prob:intro}
\begin{cases}
u_t +(\lambda+A)u=f & {\rm on}\quad \R\times\R^{d}_{+}\\
{\tr}_{\R^{d-1}}B_{j}u=g_j & {\rm on}\quad \R\times\R^{d-1}, j=1,\ldots,m,
\end{cases}
\end{equation}
where ``tr'' denotes the trace operator, $A$ is an elliptic differential operator of order $2m$, and $(B_j)$ is a family of differential operators of order $m_j<2m$ for $j=1,\ldots,m$. The leading coefficients of $A$ are assumed to be in the class of vanishing mean oscillations (VMO) both in the time and space variables, while the operators $B_j$ are assumed to have variable leading coefficients. In addition, we assume that near the boundary $(A,B_j)$ satisfies the Lopatinskii--Shapiro condition. Roughly speaking, it is an algebraic condition involving the symbols of the principal part of the operators $A$ and $B_j$ with fixed coefficients, which is equivalent to the solvability of certain systems of ordinary differential equations. See e.g. \cite{Lop53,Sha53,ADN64,WlokaBook}.

Below in Theorem \ref{thm:VMOproblemLStx}, we establish weighted $L_{p}(L_{q})$-estimates with $p,q\in(1,\infty)$ for \eqref{prob:intro} with time-dependent weights in the Muckenhoupt class. This generalizes the recent result obtained by the authors in \cite[Theorem 3.4]{DG17}, where $B_j$ is assumed to have constant leading coefficients.

In contrast to the case when $A$ has uniformly continuous leading coefficients, the extension of the results in \cite{DG17} to boundary operators with variable leading coefficients is nontrivial and does not follow from the standard perturbation argument.
In fact, under the VMO assumption on the coefficients of $A$, in the case when the boundary operators have variable leading coefficients, to apply the method of freezing the coefficients as in \cite[Lemma 4.6]{DG17} one would need to show the mean oscillation estimates of \cite[Lemma 4.5]{DG17} for a equation with inhomogeneous boundary conditions. To the best of the authors knowledge, this case is not covered by the known theory. Moreover, the well-known localization procedure (see for instance \cite[Section 8]{DHP}) does not seem to directly apply to the case $p\neq q$, since we would need a partition of unity argument in both $t$ and $x$.

Our proof is based on a preliminary result for the case $p=q$, Lemma \ref{lemma:estconstcoeff}, in which the $L_q(L_q)$-estimate is shown as a combination of a recent result of Lindemulder in \cite{Lin17} and the estimates in \cite{DG17}, as well as the available extrapolation theory (see \cite[Theorem 1.4]{CMP} and \cite[Theorem 2.5]{DK16}) to extrapolate to $p\neq q$. We also use in a crucial way a version of the Fefferman--Stein sharp function theorem with $A_p$ weights in spaces of homogeneous type, which was recently established in \cite{DK16}.

Research on $L_{p}(L_{q})$-regularity for higher-order equations as \eqref{prob:intro}  has been developed in the last decades by mainly two different approaches. On the one hand, a PDE approach has been developed by a series of papers by Krylov, Dong, and Kim. See e.g. \cite{Krypq}, \cite{DK09, DK11} and references therein. In \cite{DK11} a new technique was developed to produce mean oscillation estimates for equations in the whole and half spaces with the Dirichlet boundary condition, for $p=q$. These results had been extended by the same authors in \cite{DK16} to mixed $L_p(L_q)$-spaces with Muckenhoupt weights and small BMO assumptions on the space variable, for any $p,q\in(1,\infty)$. It is worth noting that in all these references as well as others papers in the literature, VMO coefficients were only considered for equations with specific boundary conditions (Dirichlet, Neumann, or conormal, etc.).

On the other hand, $L_{p}(L_{q})$-regularity can be viewed in a functional analytic approach as an application of a more general abstract result, namely that of maximal $L_{p}$-regularity.
Maximal $L_p$-regularity means that, under certain assumption on $g_j$, for all $f\in L_p(\R,L_q(\R^d_+))$, the solution to the evolution problem \eqref{prob:intro} has the  ``maximal'' regularity in the sense that $u_t$ and $Au$ are both in $L_p(\R,L_q(\R^d_+))$. We refer to \cite{Weis01,PS01, DHP}, \cite{HH, HHH}, \cite{GV,GVsystem} for further informations on autonomous and non-autonomous problems and applications to higher order equations.

In \cite{DHP}, Denk, Hieber, and Pr\"{u}ss obtained $L_p(L_q)$-regularity for every $p,q\in (1,\infty)$ for autonomous, operator-valued parabolic problems on the half space and on domains with homogeneous boundary conditions of the Lopatinskii--Shapiro type. The leading coefficients of the operators involved are assumed to be bounded and uniformly continuous, and their proofs combine operator sum methods with tools from vector-valued harmonic analysis. These results were generalized in \cite{DHP07} by the same authors to $L_{p}(L_{q})$-regularity for non-autonomous, operator-valued parabolic initial-boundary value problems with inhomogeneous boundary data, under the assumption that $t\rightarrow A(t)$ is continuous. See also Weidemaier \cite{Weide02} for the special case where $m=1$, the coefficients are complex-valued coefficients and $q\le p$.
Later, Meyries and Schnaubelt in \cite{MS12b} further generalized the results of \cite{DHP07} to the weighted time-dependent setting, where the weights considered are Muckenhoupt power-type weights. See also \cite{MeyThesis}. Very recently, Lindemulder in \cite{Lin17} generalized the results of \cite{MS12b} to the setting of power weights both in the time and the space variables. In all these results, the leading coefficients of the operators are assumed to be bounded and uniformly continuous in both the time and  space variables.

In this paper, we relax the assumptions on the coefficients of the operators involved to be VMO in the time and space variables, and with inhomogeneous general boundary operators having variable leading coefficients and satisfying the Lopatinskii--Shapiro condition. The main result of this paper is stated in Theorem \ref{thm:VMOproblemLStx}, and in the elliptic setting in Theorem \ref{thm:VMOellipticLStx}. As an application, we obtain in Theorem \ref{thm:VMOdomainLStx} the $L_p(L_q)$-estimates on bounded smooth domains, and we state the elliptic counterpart in Theorem \ref{thm:VMOellipticdomainLStx}. The results here presented generalize the ones in \cite{DG17}, in which the boundary operators were assumed to have constant leading coefficients and only the half space setting was considered.

The paper is organized as follows. In Section \ref{sec:preliminaries} we give the necessary preliminary results and introduce the notation. In Section \ref{sec:assumptions} we list the main assumptions on the operators and state the main theorem. In Section \ref{sec:auxres} we prove an auxiliary lemma needed for the proof of the main theorem, which is given in Section \ref{sec:proofmainresult}. Finally, in Section \ref{sec:domain} we prove $L_p(L_q)$-estimates on domains by using the estimates in the previous sections.
	
\section{Preliminaries}\label{sec:preliminaries}
In this section, we state some necessary preliminary results and introduce the notation used throughout the paper.

\subsection{$A_p$-weights}\label{subsectionDG:weight}

Details on Muckenhoupt weights can be found in \cite[Chapter 9]{GrafakosModern} and \cite[Chapter V]{SteinHA}.

A {\em weight} is a locally integrable function on $\R^d$ with $\omega(x)\in (0,\infty)$ for almost every $x \in \R^d$. The space $L_p(\R^d,\omega)$ is defined as all measurable functions $f$ with
\begin{equation*}
\|f\|_{L_p(\R^d,\omega)}=\Big(\int_{\R^d} |f|^p\ \omega \, d\mu\Big)^\frac{1}{p}<\infty \quad  \text{if $p\in [1, \infty)$},
\end{equation*}
and $\displaystyle \|f\|_{L_\infty(\R^d,\omega)} = \esssup_{x\in \R^d} |f(x)|$.

We recall the class of Muckenhoupt weights $A_{p}$ for $p \in (1,\infty)$. A weight $\omega$ is said to be an {\em $A_{p}$-weight} if
\begin{align*}
[\omega]_p=[\omega]_{A_{p}}:=\sup_{B} \Big(\dashint_B \omega(x) \, dx\Big) \Big(\dashint_B \omega(x)^{-\frac{1}{p-1}}\, dx \Big)^{p-1}<\infty.
\end{align*}
Here the supremum is taken over all balls $B\subset \R^d$ and $\avint_{B} = \frac{1}{|B|}\int_{B}$. The extended real number $[\omega]_{A_{p}}$ is called the {\em $A_p$-constant}. In the case of the half space $\R^{d}_{+}$, we replace the balls $B$ in the definition by $B\cap \R^{d}_{+}=:B^{+}$ with center in $\overline{\R^{d}_{+}}$.

The following properties will be used. For the given $\omega\in A_{p}(\R)$, an open interval $I\in \bR$ and a measurable set $E\subset I$, it holds that
\begin{equation}
                            \label{eq11.07}
\omega(E)/\omega(I)\ge C\big(|E|/|I|\big)^p,
\end{equation}
where $C>0$ is a constant depending only on $p$ and $[\omega]_{A_p}$, and $|E|$ is the Lebesgue measure of $E$.
Moreover, using a reverse H\"{o}lder's inequality (see \cite[Corollary 9.2.4 and Remark 9.2.3]{GrafakosModern}), there is a positive number $\sigma_1=\sigma_1(p,[\omega]_{p})$ such that $p-\sigma_1>1$ and
\[
\omega\in A_{p-\sigma_1}(\R).
\]
Consequently, instead of \eqref{eq11.07}, we have
\begin{equation}
                            \label{eq11.08}
\omega(E)/\omega(I)\ge C\big(|E|/|I|\big)^{p-\sigma_1},
\end{equation}

The celebrated result of Rubio de Francia (see \cite{Rubio82, Rubio83, Rubio84}, \cite[Chapter IV]{GarciaRubio}) allows one to extrapolate from weighted $L_p$-estimates for a single $p$ to weighted $L_q$-estimates for all $q$. The proofs and statement have been considerably simplified and clarified in \cite[Theorem 3.9]{CMP}. The following version of the extrapolation theorem \cite[Theorem 3.9]{CMP} will be needed. Its main feature is that, to prove \eqref{eq:extr2} for a given $\omega\in A_p, p\in(1,\infty)$, the estimate \eqref{eq:extr1} as an assumption needs to hold only for a subset of $A_{p_0}$, not for all weights in $A_{p_0}$. We refer to \cite[Theorem 2.5]{DK16} for further details.

\begin{theorem}\label{thm:extensionRubio}
Let $f,g:\R^d\rightarrow\R$ be a pair of measurable functions, $p_0,p\in(1,\infty)$, and $\omega\in A_p$. Then there exists a constant $\Lambda_0=\Lambda_0(p_0,p,[\omega]_{A_p})\geq 1$ such that if
\begin{equation}\label{eq:extr1}
\|f\|_{L_{p_0}(\tilde{\omega})}\leq C\|g\|_{L_{p_0}(\tilde{\omega})}
\end{equation}
for every $\tilde{\omega}\in A_{p_0}$ satisfying $[\tilde{\omega}]_{A_{p_0}}\leq \Lambda_0$, then we have
\begin{equation}\label{eq:extr2}
\|f\|_{L_{p}(\omega)}\leq 4C\|g\|_{L_{p}(\omega)}.
\end{equation}
\end{theorem}

\subsection{Function spaces and notation}
In this section we introduce some function spaces and notation to be use throughout the paper.

We denote $D=-i(\partial_i,\ldots,\partial_d)$ and we consider the standard multi-index notation $D^{\alpha}=D_1^{\alpha_1}\cdot\ldots\cdot D_{d}^{\alpha_d}$ and $|\alpha|=\alpha_1+\cdots+\alpha_d$ for a multi-index $\alpha=(\alpha_1,\ldots,\alpha_d)\in\N_0^d$.

Denote
$$
\R^{d}_{+}=\big\{x=(x_1,x')\in\R^d: x_1> 0,\ x'\in\R^{d-1} \big\} \quad \text{and}\quad \R^{d+1}_{+}=\R\times\R^{d}_{+}.
$$
The parabolic distance between $X=(t,x)$ and $Y=(s,y)$ in $\R^{d+1}_{+}$ is defined by $\rho(X,Y)=|x-y|+|t-s|^{\frac{1}{2m}}$.
For a function $f$ on $\mathcal{D}\subset\R^{d+1}_{+}$, we set
\[
(f)_{\mathcal{D}}=\frac{1}{|\mathcal{D}|}\int_{\mathcal{D}}f(t,x)\, dx\,dt=\dashint_{\mathcal{D}}f(t,x)\, dx\,dt.
\]
For $m=1,2,\ldots$ fixed depending on the order of the equations under consideration, we denote by
\begin{equation*}
	Q_r^{+}(t,x)=((t-r^{2m},t)\times B_r(x))\cap \R^{d+1}_+
\end{equation*}
the parabolic cylinders, where
$$
B_r(x)=\big\{y\in\R^d:|x-y|<r\big\}\subset\R^d
$$
denotes the ball of radius $r$ with center $x$.
We use $Q_r^{+}$ to indicate $Q_r^{+}(0,0)$. We also define
$$
B_r^+(x)=B_r(x)\cap \R^{d}_+.
$$
We define the mean oscillation of $f$ on a parabolic cylinder as
\[
\osc(f,Q_r^{+}(t,x)):=\dashint_{Q_r^{+}(t,x)}
\big|f(s,y)-(f)_{Q_r^{+}(t,x)}\big|\,ds\,dy
\]
and we denote for $R\in(0,\infty)$,
\[
(f)^{\sharp}_{R}:=\sup_{(t,x)\in\R^{d+1}}\sup_{r\le R}\osc(f,Q_r^{+}(t,x)).
\]
Next, we introduce the function spaces which will be used in the paper.
For $p\in (1,\infty)$ and $k\in\N_0$, we define the standard Sobolev space as
$$
W^{k}_{p}(\R^{d}_{+})=\big\{u\in L_{p}(\R^{d}_{+}):\ D^{\alpha}u\in L_{p}(\R^{d}_{+})\quad \forall |\alpha|\le k \big\}.
$$
For $p,q\in(1,\infty)$, we denote
\[
L_{p}(\R^{d+1}_{+})=L_{p}(\R;L_{p}(\R^{d}_{+}))
\]
and mixed-norm spaces
\[
L_{p,q}(\R^{d+1}_{+})=L_{p}(\R;L_{q}(\R^{d}_{+})).
\]
For parabolic equations we denote for $k=1,2,\ldots$,
\[W^{1,k}_{p}(\R^{d+1}_{+})=W^{1}_{p}(\R;L_{p}(\R^{d}_{+}))\cap L_{p}(\R;W^{k}_{p}(\R^{d}_{+}))\]
and mixed-norm spaces
\[
W^{1,k}_{p,q}(\R^{d+1}_{+})=W^{1}_{p}(\R;L_{q}(\R^{d}_{+}))\cap L_{p}(\R;W^{k}_{q}(\R^{d}_{+})).
\]

We will use the following weighted Sobolev spaces. For $\omega\in A_{p}(\R)$ we denote
\[
L_{p,q,\omega}(\R^{d+1}_{+})=L_{p}(\R,\omega;L_{q}(\R^{d}_{+}))
\]
and
\[
W^{1,k}_{p,q,\omega}(\R^{d+1}_{+})=W^{1}_{p}(\R,\omega;L_{q}(\R^{d}_{+}))\cap L_{p}(\R,\omega;W^{k}_{q}(\R^{d}_{+})),
\]
where by $f\in L_{p,q,\omega}(\R^{d+1}_{+})$ we mean
\[
\|f\|_{L_{p,q,\omega}(\R^{d+1}_{+})}:=	\bigg(\int_{\R}\bigg(\int_{\R^{d}_{+}}|f(t,x)|^{q}\,dx\bigg)^{p/q}\omega(t)\,dt\bigg)^{1/p}<\infty.
\]
\subsection{Interpolation and trace}
The following function spaces from the interpolation theory will be needed. For more information and proofs we refer the reader to \cite{MeyThesis,Tr92,Tr1}.

For $p\in (1,\infty)$ and $s=[s]+s_{\ast}\in \R_+\backslash \N_0$, where $[s]\in\N_0$, $s_{\ast}\in(0,1)$,  we define the Slobodetskii space $W^{s}_{p}$ by real interpolation as
\[
W^{s}_{p}=(W^{[s]}_{p},W^{[s]+1}_{p})_{s_{\ast},p}.
\]
For $m\in\N$, $s\in (0,1]$, and $\omega\in A_p(\R)$, we consider weighted anisotropic spaces of the form
\[
W^{s,2ms}_{p,\omega}(\R^{d+1}_{+})=W^{s}_{p}(\R,\omega;L_{p}(\R^{d}_{+}))\cap L_{p}(\R,\omega;W^{2ms}_{p}(\R^{d}_{+})).
\]

For $p\in (1,\infty)$, $q\in [1,\infty]$, $r\in\R$, $\omega\in A_p(\R^d)$, and $X$ a Banach space, we introduce the Besov space $\B^{r}_{p,q}(\R^d)$ and the weighted $X$-valued Triebel--Lizorkin space $F^{r}_{p,q}(\R^d,\omega;X)$ as follows.

Let $\Phi(\R^d)$ be the set of all sequences $(\varphi_k)_{k\geq 0}\subset\mathcal{S}(\R^d)$ such that
\[
\widehat{\varphi}_0=\widehat{\varphi},\quad \widehat{\varphi}_1(\xi)=\widehat{\varphi}(\xi/2)-\widehat{\varphi}(\xi),\quad \widehat{\varphi}_k(\xi)=\widehat{\varphi}_{1}(2^{-k+1}\xi),
\]
where $k\geq 2$, $\xi\in\R^d$, and where the Fourier transform $\widehat{\varphi}$ of the generating function $\varphi\in\mathcal{S}(\R^d)$ satisfies $0\le \widehat{\varphi}(\xi)\le 1$ for $\xi\in\R^d$ and
\[
\widehat{\varphi}(\xi)=1\quad {\rm if}\ |\xi|\le 1,\quad  \widehat{\varphi}(\xi)=0\quad {\rm if}\ |\xi|\geq 3/2.
\]

\begin{definition}\label{def:TLspace}
Given $(\varphi_k)_{k\geq 0}\in \Phi(\R^d)$, we define the \textit{Besov space} as
\[
\B^{r}_{p,q}(\R^d)=\big\{ f\in \mathcal{S}'(\R^d):\ \|f\|_{\B^{r}_{p,q}(\R^d)}:=\|(2^{kr}\mathcal{F}^{-1}(\widehat{\varphi}_{k}\hat{f}))_{k\geq 0}\|_{\ell_q(L_{p}(\R^d))}  <\infty  \big\},
\]
and the weighted \textit{$X$-valued Triebel--Lizorkin space} as
\begin{align*}
&F^{r}_{p,q}(\R^d,\omega;X)\\
&=\big\{ f\in \mathcal{S}'(\R^d,X):\ \|f\|_{F^{r}_{p,q}(\R^d,\omega;X)}:=\|(2^{kr}\mathcal{F}^{-1}(\widehat{\varphi}_{k}\hat{f}))_{k\geq 0}\|_{L_{p}(\R^d,\omega;\ell_q(X))} <\infty  \big\}.
\end{align*}
\end{definition}

Observe that $\B^{r}_{p,p}(\R^d)=F^{r}_{p,p}(\R^d)$ by Fubini's Theorem. Moreover, we have the following equivalent definition of Slobodetskii space
\[
W^{s}_{p}(\R^d)=
\begin{cases}
W^{k}_{p}(\R^d), & s=k\in\N\\
\B^{s}_{p,p}(\R^d), & s\in \R_+\backslash \N_0.
\end{cases}
\]
Later on we will consider weighted $X$-valued Triebel-Lizorkin spaces on an interval $(-\infty,T)\subset\R$. We define these spaces by restriction.
\begin{definition}
Let $T\in(-\infty,\infty]$ and let $X$ be a Banach space. For $p\in(1,\infty)$, $q\in[1,\infty)$, $\omega\in A_p(\bR)$ and $r\in\R$ we denote by $F^{r}_{p,q}((-\infty,T),\omega;X)$ the collection of all restrictions of elements of $F^{r}_{p,q}(\R,\omega;X)$ on $(-\infty,T)$. If $f\in F^{r}_{p,q}((-\infty,T),\omega;X)$ then
$$
\|f\|_{F^{r}_{p,q}((-\infty,T),\omega;X)}=\inf\|g\|_{F^{r}_{p,q}(\R,\omega;X)}
$$
where the infimum is taken over all $g\in F^{r}_{p,q}(\R,\omega;X)$ whose restriction on $(-\infty,T)$ coincides with $f$.
\end{definition}		

We will need the following spatial traces and interpolation inequalities. For full details about the proofs, we refer the reader respectively to \cite[Lemma 3.5 and Lemma 3.10]{DHP07} for the unweighted setting and to \cite[Lemma 1.3.11 and Lemma 1.3.13]{MeyThesis} where the weights considered are power-type weights in time. The restriction of power-type weights only play a role at $t=0$ in order to have a well-defined trace space. Thus in the formulation below with $t\in\R$, the power-type weight can be replaced by any weight $\omega\in A_p(\R)$; see for instance \cite[Section 6.3]{Lin17} for details.

\begin{lemma}\label{thm:MeyLemma1.3.11}
Let $p\in(1,\infty)$, $\omega\in A_p(\R)$, $m\in\N$, and $s\in (0,1]$ so that $2ms\in\N$. Then the map
\[
{\rm tr}_{x_1=0}:W^{s,2ms}_{p,\omega}(\R^{d+1}_{+})\hookrightarrow W^{s-\frac{1}{2mp},2ms-\frac{1}{p}}_{p,\omega}(\R\times\R^{d-1})
\]
is continuous.
\end{lemma}

\begin{lemma}\label{lemma:MeyLemma1.3.13}
Let $p\in(1,\infty)$, $\omega\in A_p(\R)$, and let $m\in\N$ and $s\in[0,1)$ be given. Then for every $\varepsilon>0$, for $\beta\in\N_0^n$ with $s+\frac{|\beta|}{2m}+\frac{1}{2mp}<1$, it holds that for $u\in W^{1,2m}_{p,\omega}(\R\times\R^d_+)$,
\begin{align*}
&\|{\rm tr}_{\R^{d-1}}\nabla^{\beta}u\|_{W^{s,2ms}_{p,\omega}(\R\times\R^{d-1})}\\
&\le \varepsilon\|D^{2m}u\|_{L_{p,\omega}(\R\times\R^d_+)}+\varepsilon
\|u_{t}\|_{L_{p,\omega}(\R\times\R^d_+)}
+C_{\varepsilon}\|u\|_{L_{p,\omega}(\R\times\R^d_+)}.
\end{align*}
\end{lemma}

\section{Assumptions and main result}\label{sec:assumptions}

Let $p,q\in (1,\infty)$, and $m=1,2,\ldots$. We consider a $2m$-th order elliptic differential operator $A$ given by
\[
Au=\sum_{|\alpha|\le 2m}a_{\alpha}(t,x)D^{\alpha}u,
\]
where $a_\alpha:\R\times\R^{d}_{+}\rightarrow\C$. For $j=1,\ldots,m$ and $m_j \in\{0,\ldots,2m-1\}$, we consider the boundary differential operators $B_j$ of order $m_j$ given by
\[
B_ju=\sum_{|\beta|\leq m_j}b_{j\beta}(t,x)D^{\beta}u,
\]
where $b_{j\beta}:\R\times\R^{d}_{+}\rightarrow\C$. For convenience, here and in the sequel, we denote $D_{x_{j}}=-i\frac{\partial}{\partial x_{j}}$.

We will give conditions on the operators $A$ and $B_j$ under which the $L_p(L_q)$-estimates hold for the solution to the parabolic problem
\begin{equation}\label{prob:VMOtimespaceLStx}
\begin{cases}
u_t(t,x) + (A+\lambda)u(t,x)=f(t,x) & {\rm in}\ \R\times\R^{d}_{+}\\
B_{j}u(t,x)\bigg|_{x_1=0}=g_j & {\rm on}\ \R\times\R^{d-1},\ \  j=1,\ldots,m.
\end{cases}
\end{equation}
We also consider the corresponding elliptic problem
\begin{equation}\label{prob:ellVMOtimespaceLStx}
\begin{cases}
(A+\lambda)u=f & {\rm in}\quad \R^{d}_{+}\\
B_{j}u\big|_{x_1=0}=g_j & {\rm on}\quad \R^{d-1},\ \  j=1,\ldots,m,
\end{cases}
\end{equation}
where, for the elliptic case, the coefficients of the operators involved are functions independent on $t\in\R$, i.e., defined on $\R^d_+$.

\subsection{Assumptions on $A$ and $B_j$.}

We first introduce a parameter--ellipticity condition in the sense of \cite[Definition 5.1]{DHP}. Denote
$$
A_{\sharp}(t,x,\xi)=\sum_{|\alpha|=2m}a_{\alpha}(t,x)\xi^{\alpha}
$$
to be the \textit{principal symbol} of the operator $A$.
\let\ALTERWERTA\theenumi
\let\ALTERWERTB\labelenumi
\def\theenumi{(E)$_\theta$}
\def\labelenumi{\textbf{(E)}$_\theta$}
\begin{enumerate}
\item\label{as:condell}
Let $\theta\in(0,\pi)$. For all $t\in\R$ and $x\in\R^d_+$, it holds that
\begin{equation*}
\sigma(A_{\sharp}(t,x,\xi))\subset\Sigma_{\theta},\quad  \forall\ \xi\in\R^{n},\ |\xi|=1,
\end{equation*}
for the spectrum of the operator $A_{\sharp}(t,x,\xi)$, where $\Sigma_{\theta}=\{z\in\C\backslash\{0\}:\ |\arg(z)|<\theta\}$ and $\arg : \C\backslash\{0\} \rightarrow (-\pi,\pi]$.
\end{enumerate}
\let\theenumi\ALTERWERTA
\let\labelenumi\ALTERWERTB

The following \ref{as:LScond}-condition is of the Lopatinskii--Shapiro type. Before stating it, we need to introduce some notation.

Denote by
\begin{equation}\label{eq:opH}
A^{H}(t,x,D):=\sum_{|\alpha|=2m}a_{\alpha}(t,x)D^{\alpha}\quad
\text{and}\quad
B^{H}_{j}(D):=\sum_{|\beta|=m_{j}}b_{j\beta}(t,x)D^{\beta}
\end{equation}
the \textit{principal part} of $A$ and $B_{j}$ respectively. For fixed $(t_0,x_0)\in \overline{\bR^{d+1}_+}$, consider the operators $A^H(t_0,x_0,D)$ and $B^{H}_j(t_0,x_0,D)$. Taking the Fourier transform $\mathcal{F}_{x'}$  with respect to $x'\in\R^{d-1}$ and  letting $v(x_1,\xi):=\mathcal{F}_{x'}(u(x_1,\cdot))(\xi)$, we obtain
\begin{align*}
	A^H(t_0,x_0,\xi,D_{x_1})v
	&:=\mathcal{F}_{x'}(A^H(t_0,x_0,D)u(x_1,\cdot))(\xi)\\
	&= \sum_{k=0}^{2m}\sum_{|\beta|=k}a_{k,\beta}(t_0,x_0)\xi^{\beta}D_{x_{1}}^{2m-k}v
\end{align*}
and
\begin{align*}
	B^H_{j}(t_0,x_0,\xi,D_{x_1})v
	&:=\mathcal{F}_{x'}(B^H_{j}(t_0,x_0,D)u(x_1,\cdot))(\xi)\\	&=\sum_{k=0}^{m_j}\sum_{|\gamma|=k}b_{j,k,\gamma}(t_0,x_0)\xi^{\gamma}D_{x_{1}}^{m_j-k}v.
\end{align*}

Throughout the paper, let $R_0\in (0,1]$, $K>0$, and $\theta\in (0,\pi/2)$ be fixed constants.
\let\ALTERWERTA\theenumi
\let\ALTERWERTB\labelenumi
\def\theenumi{(LS)$_\theta$}
\def\labelenumi{\textbf{(LS)$_\theta$}}
\begin{enumerate}
	\item\label{as:LScond}
For each $(h_{1},\ldots,h_{m})^{T}\in\R^{m}$, each $\xi\in\R^{d-1}$, $\displaystyle \lambda\in\overline{\Sigma}_{\pi-\theta}$, and $(t_0,x_0)\in \overline{\bR^{d+1}_+}$, such that $|\xi|+|\lambda|\neq 0$ and $x_{0,1}\le 2R_0$, the ODE problem in $\R_+$
\begin{equation*}
\begin{cases}
\lambda v+ A^H(t_0,x_0,\xi,D_{x_1})v=0,\quad  x_1 >0,\\
B^H_{j}(t_0,x_0,\xi,D_{x_1})v\bigg|_{x_1=0}=h_{j},\quad  j=1,\ldots,m
\end{cases}
\end{equation*}
admits a unique solution $v\in C^{\infty}(\R_+)$ such that $\lim_{x\rightarrow \infty}v(x)=0$.
\end{enumerate}
\let\theenumi\ALTERWERTA
\let\labelenumi\ALTERWERTB

We now introduce a regularity condition on the leading coefficients, where $\rho$ is a parameter to be specified.
\begin{assumption}[$\rho$] \label{ass:VMO}
For $|\alpha|=2m$, there exist a constant $R_0\in (0,1]$ such that $(a_{\alpha})^{\sharp}_{R_0}\le \rho$.
\end{assumption}	

Throughout the paper, we impose the following assumptions on the coefficients of $A$ and $B_j$.
\let\ALTERWERTA\theenumi
\let\ALTERWERTB\labelenumi
\def\theenumi{(A)}
\def\labelenumi{\textbf{(A)}}
\begin{enumerate}
\item\label{as:operatorALS} For the multi-index $\alpha$, the coefficients $a_{\alpha}$ are functions
$\R\times\R^{d}_{+}\rightarrow\C$, $\|a_\alpha\|_{L_\infty}\le K$,
and satisfy Assumption \ref{ass:VMO} ($\rho$) with a parameter $\rho\in (0,1)$ to be determined later.
Moreover, $A$ satisfies condition $(E)_{\theta}$.
\end{enumerate}
\let\theenumi\ALTERWERTA
\let\labelenumi\ALTERWERTB
	
\let\ALTERWERTA\theenumi
\let\ALTERWERTB\labelenumi
\def\theenumi{(B)}
\def\labelenumi{\textbf{(B)}}
\begin{enumerate}
\item\label{as:operatorBLS} The coefficients $b_{j\beta}:\R\times\R^{d}_{+}\rightarrow\C $ satisfy
$$
b_{j\beta}\in C^{\frac{2m-m_{j}}{2m},2m-m_{j}}(\R^{d+1}_{+}),\quad
\|b_{j\beta}\|_{C^{\frac{2m-m_{j}}{2m},2m-m_{j}}}\leq K,
$$
and
$$
\lim_{|t|+|x|\rightarrow \infty}b_{j\beta}(t,x)=\overline{b}_{j\beta}.
$$
The \ref{as:LScond}-condition is satisfied by $(A,\overline{B}_{j})$ for any $A\in (E)_{\theta}$, where $\overline{B}_{j}$, $j=1,\ldots,m$, are the boundary operators with coefficients $\overline{b}_{j\beta}$.
\end{enumerate}

We can now state the main result of this paper.

\begin{theorem}\label{thm:VMOproblemLStx}
Let $T\in(-\infty,\infty]$, $p,q\in(1,\infty)$ and $\omega\in A_{p}((-\infty,T))$. There exists
$$
\rho=\rho(\theta,m,d,K,p,q,[\omega]_{p},b_{j\beta})\in (0,1)
$$
such that under the assumptions \ref{as:operatorALS}, \ref{as:operatorBLS}, and \ref{as:LScond} the following hold. There exists $\lambda_0=\lambda_0(\theta,m,d,K,p,q,[\omega]_{p},R_0,b_{j\beta})\geq 1$ such that for every $\lambda\geq\lambda_0$, for $u\in W^{1}_{p}((-\infty,T),\omega;L_{q}(\R^d_+))\cap L_{p}((-\infty,T),\omega;W^{2m}_{q}(\R^d_+))$ satisfying the problem \eqref{prob:VMOtimespaceLStx}, where $f\in L_{p,q,\omega}((-\infty,T)\times \R^d_+)$ and $$
g_j\in F^{k_j}_{p,q}((-\infty,T),\omega;L_{q}(\R^{d-1}))\cap L_{p}((-\infty,T),\omega;\B_{q,q}^{2mk_j}(\R^{d-1}))
$$
with $k_j=1-m_j/(2m)-1/(2mq)$, it holds that
\begin{align}\label{eq:VMOtimespaceLSgjtx}
&\|u_t\|_{L_{p}((-\infty,T),\omega;L_{q}(\R^{d}_+))}+\sum_{|\alpha|\le 2m}\lambda^{1-\frac{|\alpha|}{2m}}\|D^{\alpha}u\|_{L_{p}
((-\infty,T),\omega;L_{q}(\R^{d}_+))}\nonumber\\
&\le C\|f\|_{L_{p}((-\infty,T),\omega;L_{q}(\R^{d}_+))} \nonumber\\
&\quad+ C\sum_{j=1}^m
\|g_j\|_{F^{k_j}_{p,q}((-\infty,T),\omega;L_{q}(\R^{d-1}))\cap L_{p}((-\infty,T),\omega;\B_{q,q}^{2mk_j}(\R^{d-1}))},
\end{align}
where $C=C(\theta,m,d,K,p,q,[\omega]_{p},b_{j\beta})>0$ is a constant.
\end{theorem}

Using the same arguments as in \cite[Theorem 3.5]{DG17}, from the a priori estimates for the parabolic equation in Theorem \ref{thm:VMOproblemLStx}, we obtain the a priori estimates for the higher-order elliptic equation as well. The key idea is that the solutions to elliptic equations can be viewed as steady state solutions to the corresponding parabolic cases. The argument is quite standard, so we omit the proof. The interested reader can find more details in \cite[Theorem 5.5]{DK16} and \cite[Theorem 2.6]{Kry07}.

We state below the elliptic version of Theorem \ref{thm:VMOproblemLStx}. In this case the coefficients of $A$ and $B_j$ are independent of $t$.

\begin{theorem}\label{thm:VMOellipticLStx}
Let $q\in(1,\infty)$. There exists
$$
\rho=\rho(\theta,m,d,K,q,b_{j\beta})\in (0,1)
$$
such that under the assumptions \ref{as:operatorALS}, \ref{as:operatorBLS}, and \ref{as:LScond}, the following hold. There exists $\lambda_0=\lambda_0(\theta,m,d,K,q,R_0,b_{j\beta})\geq 0$ such that for any $\lambda\geq\lambda_0$ and $u\in W^{2m}_{q}(\R^d_+)$ satisfying
\begin{equation*}
\begin{cases}
(A+\lambda)u=f & {\rm in}\ \R^{d}_{+}\\
B_{j}u\big|_{x_1=0}=g_j & {\rm on}\ \R^{d-1},
\end{cases}
\end{equation*} where $f\in L_{q}(\R^d_+)$ and $g_j\in \B_{q,q}^{2mk_j}(\R^{d-1})$ with $k_j=1-m_j/(2m)-1/(2mq)$, it holds that
\begin{equation*}
\sum_{|\alpha|\le 2m}\lambda^{1-\frac{|\alpha|}{2m}}\|D^{\alpha}u\|_{L_{q}(\R^{d}_+)}\le C\|f\|_{L_{q}(\R^{d}_+)}+ C\sum_{j=1}^m\|g_j\|_{\B_{q,q}^{2mk_j}(\R^{d-1})},
\end{equation*}
where $C=C(\theta,m,d,K,q,b_{j\beta})>0$ is a constant.
\end{theorem}

In the last section, we shall illustrate how to obtain from the above theorems the corresponding results for equations in bounded and smooth domains.

\section{An auxiliary result}\label{sec:auxres}
Throughout the section, we assume that $A$ and $B_j$ consist only of their principal part.

Let
$$A_0=\sum_{|\alpha|=2m}\bar a_{\alpha}D^{\alpha}$$
be an operator with constant coefficients satisfying $|\overline{a}_{\alpha}|\le K$ and the condition \ref{as:condell} with $\theta\in(0,\pi/2)$, and let
\[
\overline{B}_j=\sum_{|\beta|=m_j}\overline{b}_{j\beta}D^{\beta},
\]
where the coefficients $\overline{b}_{j\beta}$ are also constants.

We prove an auxiliary estimate, which is derived from a result in \cite{Lin17}. For a weight $\omega\in A_q(\R)$, we denote in the following $L_{q,\omega}(\R\times \R^{d}_{+}):=L_{q}(\R,\omega;L_{q}(\R^{d}_{+}))$.

\begin{lemma}\label{lemma:estconstcoeff}
Let $T\in(-\infty,+\infty]$, $q\in(1,\infty)$, and $\omega\in A_q(-\infty,T)$. Let $A_0$ and $\overline{B}_j$ be as above. Assume that for some $\theta\in(0,\pi/2)$,  $(A_0,\overline{B}_{j})$ satisfies the \ref{as:LScond}-condition.
Then for every $f\in L_{q,\omega}((-\infty,T)\times \R^{d}_{+})$ and
$$
g_j\in W^{k_j,2mk_j}_{q,\omega}((-\infty,T)\times\R^{d-1})
$$
with $j\in\{1,\ldots,m\}$, $m_j \in\{0,\ldots,2m-1\}$, $k_j=1-m_j/(2m)-1/(2mq)$ and $u\in W^{1,2m}_{q,\omega}((-\infty,T)\times \R^{d}_{+})$ satisfying
\begin{equation}\label{prob:lemmaestconstcoeff}
\begin{cases}
u_t(t,x) + (\lambda+A_0)u(t,x)=f(t,x) & {\rm in}\ (-\infty,T)\times\R^{d}_{+}\\
\overline{B}_{j}u(t,x)\big|_{x_1=0}=g_j(t,x) & {\rm on}\ (-\infty,T)\times\R^{d-1},
\end{cases}
\end{equation}
with $\lambda\ge 0$, we have
\begin{multline}\label{eq:lemmaestconstcoeff}
\|u_t\|_{L_{q,\omega}((-\infty,T)\times \R^{d}_{+})}+\sum_{|\alpha|\le 2m}\lambda^{1-\frac{|\alpha|}{2m}}\|D^{\alpha}u\|_{L_{q,\omega}((-\infty,T)\times \R^{d}_{+})}\\
\le C \|f\|_{L_{q,\omega}((-\infty,T)\times \R^{d}_{+})}
+ C\sum_{j=1}^{m}\|g_j\|_{W^{k_j,2mk_j}_{q,\omega}((-\infty,T)\times\R^{d-1})},
\end{multline}
with $C=C(\theta,m,d,K,q,b_{j\beta},[\omega]_{q})>0$. Moreover, for any $\lambda>0$,
$$
f\in L_{q,\omega}((-\infty,T)\times \R^{d}_{+})\quad \text{and}\quad
g_j\in W^{k_j,2mk_j}_{q,\omega}((-\infty,T)\times\R^{d-1})
$$
with $j$, $m_j$, and $k_j$ as above, there exists a unique solution $u\in W^{1,2m}_{q,\omega}((-\infty,T)\times \R^{d}_{+})$ to \eqref{prob:lemmaestconstcoeff}.
\end{lemma}

\begin{proof}
Consider first  $T=\infty$. For any $\omega\in A_q(\R)$, let $u\in W^{1,2m}_{q,\omega}(\R_+\times \R^{d}_{+})$ be a solution to \eqref{prob:lemmaestconstcoeff}.

Decompose $u=v+w$, where:
\begin{itemize}
\item $w\in W^{1,2m}_{q,\omega}(\R^{d+1}_{+})$ is the solution to the inhomogeneous problem
\begin{equation}\label{prob:VMOtimespaceW}
\begin{cases}
w_t + (A_0+\lambda)w=f & {\rm in}\ \R\times\R^{d}_{+}\\
\overline{B}_{j} w\big|_{x_{1}=0}=0 & {\rm on}\ \partial\R^{d+1}_{+},\ j=1,\ldots,m\\
\end{cases}
\end{equation}
\item $v\in W^{1,2m}_{q,\omega}(\R^{d+1}_{+})$ is the solution to the homogeneous problem
\begin{equation}\label{prob:VMOtimespaceV}
\begin{cases}
v_t + (A_0+\lambda)v=0 & {\rm in}\ \R\times\R^{d}_{+}\\
\overline{B}_{j} v\big|_{x_1=0}=g_j(t,x) & {\rm on}\ \R\times\R^{d-1},\ \  j=1,\ldots,m.
\end{cases}
\end{equation}
\end{itemize}

It follows directly from \cite[Theorem 3.4 (i)]{DG17} with $p=q$ that the solution $w\in W^{1,2m}_{q,\omega}(\R^{d+1}_+)$ of \eqref{prob:VMOtimespaceW} satisfies
\begin{equation}\label{eq:estw}
\|w_{t}\|_{L_{q,\omega}(\R^{d+1}_{+})}+\sum_{|\alpha|\le 2m}\lambda^{1-\frac{|\alpha|}{2m}}\|D^{\alpha}w\|_{L_{q,\omega}(\R^{d+1}_{+})}\le C\|f\|_{L_{q,\omega}(\R^{d+1}_+)},
\end{equation}
with $C=C(\theta,K,d,m,q,b_{j\beta},[\omega]_{q})$.

Consider now \eqref{prob:VMOtimespaceV}. Since $A_0$ and $\overline{B}_{j}$ have constant coefficients,  using a scaling $t\rightarrow \lambda^{-1} t$, $x\rightarrow \lambda^{-1/2m}x$, for a general $\lambda\in (0,1)$, we get that $\tilde{v}(t,x):=v(\lambda^{-1}t,\lambda^{-1/2m}x)$ satisfies
\begin{equation}\label{prob:lemmadhpscaling}
\begin{cases}
\tilde{v}_t (t,x) + (1+A_0)\tilde{v}(t,x)=0 & {\rm in}\ \R\times\R^{d}_{+}\\
\overline{B}_{j}\tilde{v}(t,x)\big|_{x_1=0}=\tilde{g}_{j}(t,x) & {\rm on}\ \R\times\R^{d-1},
\end{cases}
\end{equation}
where
$$
\tilde{g}_j(t,x)=\lambda^{-m_j/2m}g_{j}(\lambda^{-1}t,\lambda^{-1/2m}x).
$$
Note that $\tilde \omega(t):=\omega(\lambda^{-1}t)\in A_q(\bR)$ and $[\tilde \omega]_q=[\omega]_q$.
Applying \cite[Lemma 6.6]{Lin17} to \eqref{prob:lemmadhpscaling} with $p=q$ and $\gamma=0$, we get that the solution $\tilde{v}\in W^{1,2m}_{q,\omega}(\R^{d+1}_+)$ to \eqref{prob:lemmadhpscaling} satisfies
\begin{equation*}
\|\tilde{v}_t\|_{L_{q,\tilde\omega}(\R^{d+1}_{+})}
+\|D^{2m}\tilde{v}\|_{L_{q,\tilde\omega}(\R^{d+1}_{+})}
\le C\sum_{j=1}^{m}\|\tilde{g}_j\|_{W^{k_j,2mk_j}_{q,\tilde \omega}(\R\times\R^{d-1})},
\end{equation*}
with $C=C(\theta,m,d,K,q,b_{j\beta},[\omega]_{q})$. We remark that although the estimate is not explicitly stated in this reference, it can be extracted from the proof there.

Now scaling back and using Definition \ref{def:TLspace}, it is easily seen that
\begin{equation*}
\|v_t\|_{L_{q,\omega}(\R^{d+1}_{+})}+\|D^{2m}v\|_{L_{q,\omega}(\R^{d+1}_{+})}
\le  C\sum_{j=1}^{m}\|g_j\|_{W^{k_j,2mk_j}_{q,\omega}(\R\times\R^{d-1})},
\end{equation*}
with the constant $C$ independent of $\lambda\in (0,1)$.
Sending $\lambda\rightarrow 0$, we obtain that the above estimate holds when $\lambda=0$.
Finally, by applying an argument of S. Agmon as in \cite[Theorem 4.1]{Kry07}, from the above estimate with $\lambda=0$ it follows that when $\lambda>0$,
\begin{multline}\label{eq:estagmonv}
\|v_t\|_{L_{q,\omega}(\R^{d+1}_{+})}+\sum_{|\alpha|\leq 2m}\lambda^{1-\frac{|\alpha|}{2m}}\|D^{\alpha}v\|_{L_{q,\omega}(\R^{d+1}_{+})}\le  C\sum_{j=1}^{m}\|g_j\|_{W^{k_j,2mk_j}_{q,\omega}(\R\times\R^{d-1})},
\end{multline}
with constant $C=C(\theta,m,d,K,q,b_{j\beta},[\omega]_{q})$.
Since $u=w+v$, by \eqref{eq:estw} and \eqref{eq:estagmonv} we get \eqref{eq:lemmaestconstcoeff} with $T=\infty$ and $C=C(\theta,m,d,K,q,b_{j\beta},[\omega]_{q})>0$. The solvability follows directly by the solvability argument in \cite[Section 6]{DG17}, or the argument in \cite[Lemma 6.6]{Lin17}.

The proof for $T<\infty$ follows now the lines of \cite[Lemma 4.1]{DG17}, so we omit the details.
\end{proof}

\section{Proof of Theorem \ref{thm:VMOproblemLStx}}\label{sec:proofmainresult}
The proof of Theorem  \ref{thm:VMOproblemLStx} is divided into several steps. From Steps 1 to 3, we will assume $p=q\in(1,\infty)$ and we will show that the estimate \eqref{eq:VMOtimespaceLSgjtx} holds in this case. In Step 4, we will extrapolate the estimate from the previous steps to the case $p\neq q$ and complete the proof.

\begin{proof}[Proof of Theorem \ref{thm:VMOproblemLStx}]
It suffices to consider $T=\infty$. For the general case when $T\in (-\infty,\infty]$, we can follow the proof of \cite[Lemma 4.1]{DG17}, so we omit the details.

Recall that the lower-order coefficients in $A$ are bounded by $K$. By moving the terms $a_{\alpha}(t,x)D^{\alpha}u$ with $|\alpha|<2m$ to the right-hand side of the equation and taking a sufficiently large $\lambda$,
we may assume the lower-order coefficients of $A$ to be all zero.

Denote
$$
\tilde Q_r(t_0,x_0)=[t_0-r^{2m},t_0+r^{2m})\times C^+_{2r}(x_0),
$$
where
$C^+_{2r}(x_0)$ denotes a cube centered in $x_0$ having side-length $2r$
and axes parallel to the coordinate axes, intersected with the half space $\R^d_+$.

Let $R$ be a large constant to be specified.

\emph{Step 1.} We first consider the case $p=q$. We assume that there exists a constant $\Lambda_0\geq 1$ such that $[\omega]_{q}\leq \Lambda_0$ and we assume that $u$ is supported in $\overline{\R^{d+1}_+}\setminus \overline{\tilde Q_R}$.  Fix a point $(t_0,x_0)\in \overline{\R^{d+1}_+}\setminus \overline{\tilde Q_R}$ and set
\[
A(t_0,x_0)u=\sum_{|\alpha|=2m}a_{\alpha}(t_0,x_0)D^{\alpha}u,
\] Decompose $u=u_1+u_2$, where $u_1$ is a solution to
\begin{equation}
\label{eq4.55b}
\begin{cases}
\partial_t u_1+(\lambda+A(t_0,x_0))u_1=0 & {\rm in}\ \R^{d+1}_{+}\\
\sum_{|\beta|=m_j}\bar b_{j\beta}D^{\beta}u_1=- \sum_{|\beta|<m_j}b_{j\beta}(t,x)D^{\beta}u \\
\quad+ \sum_{|\beta|=m_j}(\bar b_{j\beta}-b_{j\beta}(t,x))D^{\beta}u +  g_j& {\rm on}\ \partial \R^{d+1}_{+},
\end{cases}
\end{equation}
and $u_2$ is a solution to
\begin{equation}\label{eq4.57b}
\begin{cases}
\partial_t u_2+(\lambda+A)u_2=f-(A-A(t_0,x_0))u_1 & {\rm in}\ \R^{d+1}_{+}\\
\sum_{|\beta|=m_j}\bar b_{j\beta}D^{\beta}u_2=0& {\rm on}\ \partial\R^{d+1}_{+}.
\end{cases}
\end{equation}

By Lemma \ref{lemma:estconstcoeff}, we first solve \eqref{eq4.55b}. It follows from Lemma \ref{thm:MeyLemma1.3.11} that
\begin{align*}
&\|\partial_t u_1\|_{L_{q,\omega}(\R^{d+1}_{+})}+\sum_{|\alpha|\leq 2m}\lambda^{1-\frac{|\alpha|}{2m}}\|D^{\alpha}u_1\|_{L_{q,\omega}(\R^{d+1}_{+})}\\
&\leq C\Big\|\sum_{|\beta|<m_j}b_{j\beta}D^{\beta}u\Big\|_{W^{\frac{2m-m_j}{2m},2m-m_j}_{q,\omega}(\R^{d+1}_{+})}+ C\sum_{j=1}^m\|g_j\|_{W^{k_j,2mk_j}_{q,\omega}(\partial \R^{d+1}_{+})}\\
&\ \ \ +C\Big\|\sum_{|\beta|=m_j}(\overline{b}_{j\beta}-b_{j\beta})D^{\beta}u\Big\|_{W^{\frac{2m-m_j}{2m},2m-m_j}_{q,\omega}(\R^{d+1}_{+})}.
\end{align*}
Since $b_{j\beta}(t,x)\rightarrow\overline{b}_{j\beta}$ for
$|t|+|x|\rightarrow\infty$, given $\varepsilon>0$ and taking $R>0$ large enough it holds that
\[
\sup_{(t,x)\in \overline{\R^{d+1}_+}\setminus \overline{\tilde Q_R}}|\overline{b}_{j\beta}-b_{j\beta}(t,x)|<\varepsilon.
\]
This and the interpolation lemma \ref{lemma:MeyLemma1.3.13} yields that
\begin{equation}\label{eq:prob4.55b}
\begin{aligned}
&\|\partial_t u_1\|_{L_{q,\omega}(\R^{d+1}_{+})}+\sum_{|\alpha|\leq 2m}\lambda^{1-\frac{|\alpha|}{2m}}\|D^{\alpha}u_1\|_{L_{q,\omega}(\R^{d+1}_{+})}\\
&\leq C\sum_{j=1}^m\|g_j\|_{W^{k_j,2mk_j}_{q,\omega}(\partial \R^{d+1}_{+})} + C_{K}\Big(\varepsilon\|D^{2m}u\|_{L_{q,\omega}(\R^{d+1}_{+})}+\varepsilon\|u_{t}\|_{L_{q,\omega}(\R^{d+1}_{+})}\\
&\ \ \ + C_{\varepsilon}\|u\|_{L_{q,\omega}(\R^{d+1}_{+})}\Big).
\end{aligned}
\end{equation}
Now, \cite[Theorem 3.4]{DG17} with $w=1$ applied to \eqref{eq4.57b} yields that for $\lambda\geq \lambda_0$, where $\lambda_0>0$ depends only on the constant $C_2$ from \cite[Proposition 5.2]{DG17}, it holds that
\begin{equation}\label{eq:prob4.57b}
\begin{aligned}
&\|\partial_t u_2\|_{L_{q,\omega}(\R^{d+1}_{+})}+\sum_{|\alpha|\leq 2m}\lambda^{1-\frac{|\alpha|}{2m}}\|D^{\alpha}u_2\|_{L_{q,\omega}(\R^{d+1}_{+})} \\
&\leq C\|f\|_{L_{q,\omega}(\R^{d+1}_{+})}+C\|(A-A(t_0,x_0))u_1\|_{L_{q,\omega}(\R^{d+1}_{+})}\\
&\leq C\|f\|_{L_{q,\omega}(\R^{d+1}_{+})}+C_{K}\sum_{|\alpha|\leq 2m}\|D^{\alpha}u_1\|_{L_{q,\omega}(\R^{d+1}_{+})},
\end{aligned}
\end{equation}
provided that $\rho\le \bar\rho$, where $\bar\rho>0$ is a constant depending only on $\theta$, $m$, $d$, $K$, $q$, $[\omega]_q$, and $\bar b_{j\beta}$.
Since $u=u_1+u_2$, by \eqref{eq:prob4.55b} and  \eqref{eq:prob4.57b}, it follows that
\begin{align*}
&\|u_t\|_{L_{q,\omega}(\R^{d+1}_{+})}+\sum_{|\alpha|\leq 2m}\lambda^{1-\frac{|\alpha|}{2m}}\|D^{\alpha}u\|_{L_{q,\omega}(\R^{d+1}_{+})}\\
&\leq \|\partial_t u_1\|_{L_{q,\omega}(\R^{d+1}_{+})}+\sum_{|\alpha|\leq 2m}\lambda^{1-\frac{|\alpha|}{2m}}\|D^{\alpha}u_1\|_{L_{q,\omega}(\R^{d+1}_{+})}\\
&\ \ \ +\|\partial_t u_2\|_{L_{q,\omega}(\R^{d+1}_{+})}+\sum_{|\alpha|\leq 2m}\lambda^{1-\frac{|\alpha|}{2m}}\|D^{\alpha}u_2\|_{L_{q,\omega}(\R^{d+1}_{+})}\\
&\leq \|\partial_t u_1\|_{L_{q,\omega}(\R^{d+1}_{+})}+\sum_{|\alpha|\leq 2m}\lambda^{1-\frac{|\alpha|}{2m}}\|D^{\alpha}u_1\|_{L_{q,\omega}(\R^{d+1}_{+})}\\
&\ \ \ +C_{K}\sum_{|\alpha|\leq 2m}\|D^{\alpha}u_1\|_{L_{q,\omega}(\R^{d+1}_{+})}+C\|f\|_{L_{q,\omega}(\R^{d+1}_{+})}\\
&\leq C\|f\|_{L_{q,\omega}(\R^{d+1}_{+})}+C_{K}\|g_j\|_{W^{k_j,2mk_j}_{q,\omega}(\partial \R^{d+1}_{+})}\\ &\ \ \ + C_{K}(\varepsilon\|D^{2m}u\|_{L_{q,\omega}(\R^{d+1}_{+})}
+\varepsilon\|u_{t}\|_{L_{q,\omega}(\R^{d+1}_{+})}
+C_{\varepsilon}\|u\|_{L_{q,\omega}(\R^{d+1}_{+})}).
\end{align*}
Now taking $\varepsilon$ small enough so that $C_{K}\varepsilon \leq 1/2$ and $\lambda\geq \bar\lambda:=\max\{\lambda_0,2C_{K}C_{\varepsilon}\}$, we get \eqref{eq:VMOtimespaceLSgjtx} for $u$ with support in $\overline{\R^{d+1}_+}\setminus \overline{\tilde Q_R}$.

\emph{ Step 2.} Let $\varepsilon$ be a small constant to be specified.
For any $(t_0,x_0)\in \overline{\tilde Q_{R+1}}$, by the stability of the \ref{as:LScond}-condition (see for instance \cite[Remark 7.10]{DHP}) and the continuity of $b_{j\beta}$, there exists $r_{t_0,x_0}\in (0,R_0)$ such that the \ref{as:LScond}-condition is satisfied by $(A(t,x),B_j(t_0,x_0))$ for any $(t,x)\in \overline{\tilde Q_{r_{t_0,x_0}}(t_0,x_0)}$ and \[
\sup_{(t,x)\in\overline{\tilde{Q}_{r_{t_0,x_0}}(t_0,x_0)}}|b_{j\beta}(t_0,x_0)-b_{j\beta}(t,x)|<\varepsilon.
\]
Assume that $u$ is supported on $\tilde Q_{\kappa_{t_0,x_0}^{-2}  r_{t_0,x_0}}(t_0,x_0)$, where $\kappa_{t_0,x_0}$ is a large constant to be determined later. We only focus on the case when $x_0^1\le R_0$. The interior case $x_0^1>R_0$, follows directly by \cite[Section 5]{DK11}, since in this case there are no boundary conditions involved.

 Similarly, we decompose $u=u_1+u_2$, where $u_1$ is a solution to
\begin{equation}
\label{eq4.55}
\begin{cases}
\partial_t u_1+(\lambda+A(t_0,x_0))u_1=0 & {\rm in}\ \R^{d+1}_{+}\\
\sum_{|\beta|=m_j}b_{j\beta}(t_0,x_0)D^{\beta}u_1=- \sum_{|\beta|<m_j}b_{j\beta}(t,x)D^{\beta}u \\
\quad+ \sum_{|\beta|=m_j}(b_{j\beta}(t_0,x_0)-b_{j\beta}(t,x))D^{\beta}u +  g_j& {\rm on}\ \partial \R^{d+1}_{+},
\end{cases}
\end{equation}
and $u_2$ is a solution to
\begin{equation}\label{eq4.57}
\begin{cases}
\partial_t u_2+(\lambda+A)u_2=f-(A-A(t_0,x_0))u_1 & {\rm in}\ \R^{d+1}_{+}\\
\sum_{|\beta|=m_j}b_{j\beta}(t_0,x_0)D^{\beta}u_2=0& {\rm on}\ \partial\R^{d+1}_{+}.
\end{cases}
\end{equation}

By Lemma \ref{lemma:estconstcoeff}, we first solve \eqref{eq4.55}. It follows from Lemma \ref{thm:MeyLemma1.3.11} that
\begin{equation}\label{eq:prob4.55}
\begin{aligned}
&\|\partial_t u_1\|_{L_{q,\omega}(\R^{d+1}_{+})}+\sum_{|\alpha|\leq 2m}\lambda^{1-\frac{|\alpha|}{2m}}\|D^{\alpha}u_1\|_{L_{q,\omega}(\R^{d+1}_{+})}\\
&\leq C\Big\|\sum_{|\beta|<m_j}b_{j\beta}D^{\beta}u\Big\|_{W^{\frac{2m-m_j}{2m},2m-m_j}_{q,\omega}(\R^{d+1}_{+})}+ C\sum_{j=1}^m\|g_j\|_{W^{k_j,2mk_j}_{q,\omega}(\partial \R^{d+1}_{+})}\\
&\ \ \ +C\Big\|\sum_{|\beta|=m_j}(b_{j\beta}(t_0,x_0)-b_{j\beta})D^{\beta}u\Big\|_{W^{\frac{2m-m_j}{2m},2m-m_j}_{q,\omega}(\R^{d+1}_{+})}\\
&\leq C\sum_{j=1}^m\|g_j\|_{W^{k_j,2mk_j}_{q,\omega}(\partial \R^{d+1}_{+})} + C_{K}\Big(\varepsilon\|D^{2m}u\|_{L_{q,\omega}(\R^{d+1}_{+})}+\varepsilon\|u_{t}\|_{L_{q,\omega}(\R^{d+1}_{+})}\\
&\ \ \ + C_{\varepsilon}\|u\|_{L_{q,\omega}(\R^{d+1}_{+})}\Big).
\end{aligned}
\end{equation}
where in the last estimate we used the interpolation Lemma \ref{lemma:MeyLemma1.3.13} and the smoothness assumption of the coefficients $b_{j\beta}$.

In order to deal with \eqref{eq4.57}, we exploit the property that $u$ has a small support. We shall first establish mean oscillation estimates in
$$
\cX:=\tilde Q_{\kappa^{-1}_{t_0,x_0} r_{t_0,x_0}}(t_0,x_0).
$$
For this, we take a dyadic decomposition of $\cX$ given by
\begin{align*}
\C_n=\big\{&Q^n=Q^n_i:\ i=(i_0,i_1,\ldots,i_d)\in\Z^{d+1},\\
&\quad i_0=0,\ldots,2^{nm}-1,i_k=0,\ldots,2^n-1
,k=1,\ldots,d\big\},
\end{align*}
where $n\in\Z$, and for $x_0=(x_0^1, x_0 ')\in\R_+\times\R^{d-1}$ and $r_0:=\kappa_{t_0,x_0}^{-1}r_{t_0,x_0}$,
\begin{align*}
&Q^n_i=\{t_0-r_0^{2m}+r_0^{2m}2^{-nm+1}([0,1)+i_0)\}\\
&\times
\big\{\max(0,x_0^1-r_0)+\min(r_0,x_0^1)2^{-n}
([0,1)+i_1)\big\}\\
&\times
\big\{x_0'-r_0(1,\ldots,1)+r_02^{-n+1}([0,1)^{d-1}+(i_2,\ldots,i_d))\big\}.
\end{align*}
Then for each $X\in\cX$ and $Q^n\in C_n$ such that $X\in Q^n$, one can find $X_0 \in\cX$ and the smallest $r\in(0,R_0)$ such that $Q^n\subset Q^{+}_{r}(X_0)$ and
\begin{equation}\label{eq:diadycest}
\dashint_{Q_n}|f(Y)-f_{|_{n}}(X)|dY\leq C\ \dashint_{Q^{+}_{r}(X_0)}|f(Y)-(f)_{Q^{+}_r(X_0)}| dY,
\end{equation}
with $C=C(d,m)$, where $f_{|_{n}}(X)=\dashint_{Q_n}f(Y)dY$.
For $x\in\cX$, we define the dyadic sharp function of $f$ by
\begin{align*}
f^{\sharp}_{dy}(x)&:=\sup_{n<\infty}\dashint_{Q^n\ni x}|f(y)-f_{|_{n}}|(x)dy.
\end{align*}

Recall that there is a positive number $\sigma_1=\sigma_1(q,[\omega]_{q})$ such that $q-\sigma_1>1$ and
\[
\omega\in A_{q-\sigma_1}(\R).
\]
We take $q_0,\mu\in (1,q)$ satisfying $\displaystyle q_0 \mu=\frac{q}{q-\sigma_1}>1$.
Then it holds that
\begin{equation}\label{eq:weightinclusion}
\omega\in A_{q-\sigma_1}(\R)= A_{q/(q_0 \mu)}(\R)\subset A_{q/q_0}(\R).
\end{equation}
By \eqref{eq:diadycest} and the mean oscillation estimates of \cite[Lemma 4.6]{DG17} with $\kappa=\kappa_{t_0,x_0}$ and $\mu,\varsigma$ satisfying $\frac{1}{\mu}+\frac{1}{\varsigma}=1$,
\begin{align*}
&\dashint_{Q_n}|\partial_t u_2(Y)-(\partial_t u_2)_{|n}(X)|dY+\sum_{\substack{|\alpha|\leq 2m\\ \alpha_1<2m}}\lambda^{1-\frac{|\alpha|}{2m}}
\dashint_{Q_n}|D^{\alpha}u_2(Y)-(D^{\alpha}u_2)_{|n}(X)|dY
\\
&\leq C\kappa_{t_0,x_0}^{-(1-\frac{1}{q_0})}\sum_{|\alpha|\leq 2m}\lambda^{1-\frac{|\alpha|}{2m}}(|D^{\alpha}u_2|^{q_0})^{\frac{1}{q_0}}_{Q^{+}_{\kappa_{t_0,x_0} r}(X_0)}+C\kappa_{t_0,x_0}^{\frac{d+2m}{q_0}}(|h|^{q_0})^{\frac{1}{q_0}}_{Q^{+}_{\kappa_{t_0,x_0} r}(X_0)}\\
&\quad +C\kappa_{t_0,x_0}^{\frac{d+2m}{q_0}}\rho^{\frac{1}{q_0\varsigma}}(|D^{2m}u_2|^{q\mu})^{\frac{1}{q_0\mu}}_{Q^{+}_{\kappa_{t_0,x_0} r}(X_0)},
\end{align*}
where $h:=f-(A-A(t_0,x_0))u_1$.
Taking the supremum with respect to all $Q^n\ni X$, $n\in\Z$, we see that for all $X\in\cX$,
\begin{equation}\label{eq:ineqbeforenorms}
\begin{aligned}
&(\partial_{t}u_2)^{\sharp}_{dy}(X)+\sum_{\substack{|\alpha|\leq 2m, \alpha_1<2m}}\lambda^{1-\frac{|\alpha|}{2m}}(D^{\alpha}u_2)^{\sharp}_{dy}(X)\\
&\ \leq C\kappa_{t_0,x_0}^{-(1-\frac{1}{q_0})}\sum_{|\alpha|\leq 2m}\lambda^{1-\frac{|\alpha|}{2m}}[\mathcal{M}(|D^{\alpha}u_2|^{q_0})(X)]^{\frac{1}{q_0}}\\
&\ \ \ +C\kappa_{t_0,x_0}^{\frac{d+2m}{q_0}}[\mathcal{M}(|h|^{q_0})(X)]^{\frac{1}{q_0}}+C\kappa_{t_0,x_0}^{\frac{d+2m}{q_0}}\rho^{\frac{1}{q_0 \varsigma}}[\mathcal{M}(|D^{2m}u_2|^{q_0\mu})(X)]^{\frac{1}{q_0\mu}}.
\end{aligned}
\end{equation}

By taking the $L_{q,\omega}(\R^{d+1}_{+})$-norms on both sides of \eqref{eq:ineqbeforenorms} and applying Theorem 2.3 of \cite{DK16} with $p=q$, we get for $C=C(\theta,d,m,K,q,[\omega]_{q},b_{j\beta},t_0,x_0)$,
\begin{align*}
&\|\partial_t u_2\|_{L_{q,\omega}(\cX)}+\sum_{\substack{|\alpha|\leq 2m, \alpha_1<2m}}\lambda^{1-\frac{|\alpha|}{2m}}\|D^{\alpha}u_2\|_{L_{q,\omega}(\cX)}\\
&\leq C|I|^{-1}(\omega(I))^{1/q}\bigg(\|\partial_t u_2\|_{L_{1}(\cX)}+\sum_{\substack{|\alpha|\leq 2m, \alpha_1<2m}}\lambda^{1-\frac{|\alpha|}{2m}}\|D^{\alpha}u_2\|_{L_{1}(\cX)}\bigg)\\
&\quad +C\kappa_{t_0,x_0}^{-(1-\frac{1}{q_0})}\sum_{|\alpha|\leq 2m}\lambda^{1-\frac{|\alpha|}{2m}}\|D^{\alpha}u_2\|_{L_{q,\omega}(\bR^{d+1}_+)}
+C\kappa_{t_0,x_0}^{\frac{d+2m}{q_0}}\|h\|_{L_{q,\omega}(\bR^{d+1}_+)}\\
&\quad +C\kappa_{t_0,x_0}^{\frac{d+2m}{q_0}}\rho^{\frac{1}{q_0\varsigma}}
\|D^{2m}u_2\|_{L_{q,\omega}(\bR^{d+1}_+)},
\end{align*}
where $I:=(t_0-r_{t_0,x_0}^{2m},t_0+r_{t_0,x_0}^{2m})$ and we used \eqref{eq:weightinclusion} and the weighted Hardy-Littlewood maximal function theorem to get, for instance,
\begin{equation*}
\begin{aligned}
&\|[\mathcal{M}(D^{2m}u_2)^{q_0\mu}]^{\frac{1}{q_0\mu}}\|_{L_{q,\omega}(\cX)}
=\|\mathcal{M}(D^{2m}u_2)^{q_0\mu}\|^{\frac{1}{q_0\mu}}_{L_{q/(q_0\mu),\omega}(\R^{d+1}_{+})}\\
&\le C\|(D^{2m}u_2)^{q_0\mu}\|^{\frac{1}{q_0\mu}}_{L_{q/(q_0\mu),\omega}(\R^{d+1}_{+})}
=C\|D^{2m}u_2\|_{L_{q,\omega}(\R^{d+1}_{+})},
\end{aligned}
\end{equation*}
with $C=C(d,q,[\omega]_{q})$.
Since
\[
a_{\tilde{\alpha}\tilde{\alpha}}(t,x)D^{2m}_{x_1}u_2=h-\partial_{t}u_2-\sum_{|\alpha|= 2m,\alpha_{1}<2m}a_{\alpha}(t,x)D^{\alpha}u_2-\lambda u_2,
\]
where $\tilde{\alpha}=(m,0,\ldots,0)$, we get
\begin{align}\label{eq1.44}
&\|\partial_t u_2\|_{L_{q,\omega}(\cX)}+\sum_{|\alpha|\leq 2m}\lambda^{1-\frac{|\alpha|}{2m}}\|D^{\alpha}u_2\|_{L_{q,\omega}(\cX)}\nonumber\\
&\leq C|I|^{-1}(\omega(I))^{1/q}\bigg(\|\partial_t u_2\|_{L_{1}(\cX)}+\sum_{|\alpha|\leq 2m}\lambda^{1-\frac{|\alpha|}{2m}}\|D^{\alpha}u_2\|_{L_{1}(\cX)}\bigg)\nonumber\\
&\quad +C\kappa_{t_0,x_0}^{-(1-\frac{1}{q_0})}\sum_{|\alpha|\leq 2m}\lambda^{1-\frac{|\alpha|}{2m}}\|D^{\alpha}u_2\|_{L_{q,\omega}(\bR^{d+1}_+)}
+C\kappa_{t_0,x_0}^{\frac{d+2m}{q_0}}\|h\|_{L_{q,\omega}(\bR^{d+1}_+)}\nonumber\\
&\quad +C\kappa_{t_0,x_0}^{\frac{d+2m}{q_0}}\rho^{\frac{1}{q_0\varsigma}}
\|D^{2m}u_2\|_{L_{q,\omega}(\bR^{d+1}_+)}.
\end{align}

Because $u_2=u-u_1$, by applying the triangle inequality and H\"older's inequality, we estimate the first term on the right-hand side of \eqref{eq1.44} by
\begin{align}
                        \label{eq1.47}
&C|I|^{-1}(\omega(I))^{1/q}\bigg(\|\partial_t u_2\|_{L_{1}(\cX)}+\sum_{|\alpha|\leq 2m}\lambda^{1-\frac{|\alpha|}{2m}}\|D^{\alpha}u_2\|_{L_{1}(\cX)}\bigg)\nonumber\\
&\le C|I|^{-1}(\omega(I))^{1/q}\bigg(\|\partial_t u\|_{L_{1}(\cX)}+\sum_{|\alpha|\leq 2m}\lambda^{1-\frac{|\alpha|}{2m}}\|D^{\alpha}u\|_{L_{1}(\cX)}\nonumber\\
&\quad +\|\partial_t u_1\|_{L_{1}(\cX)}+\sum_{|\alpha|\leq 2m}\lambda^{1-\frac{|\alpha|}{2m}}\|D^{\alpha}u_1\|_{L_{1}(\cX)}\bigg)\nonumber\\
&\le C|I|^{-1}(\omega(I))^{1/q}|I_1|(\omega(I_1))^{-1/q}
\bigg(\|\partial_t u\|_{L_{q,\omega}(\cX)}\nonumber\\
&\quad +\sum_{|\alpha|\leq 2m}\lambda^{1-\frac{|\alpha|}{2m}}\|D^{\alpha}u\|_{L_{q,\omega}(\cX)}\bigg)
+C\|\partial_t u_1\|_{L_{q,w}(\cX)}\nonumber\\
&\quad +C\sum_{|\alpha|\leq 2m}\lambda^{1-\frac{|\alpha|}{2m}}\|D^{\alpha}u_1\|_{L_{q,w}(\cX)},
\end{align}
where
$$
I_1:=(t_0-(\kappa_{t_0,x_0}^{-2}r_{t_0,x_0})^{2m},
t_0+(\kappa_{t_0,x_0}^{-2}r_{t_0,x_0})^{2m}),
$$
and we used the fact that $u$ is supported on $\tilde Q_{\kappa_{t_0,x_0}^{-2}  r_{t_0,x_0}}(t_0,x_0)$.
Using \eqref{eq11.08},
\begin{equation}\label{eq2.49}
|I|^{-1}(\omega(I))^{\frac 1 q}|I_1|(\omega(I_1))^{-\frac 1 q}
\le C(|I_1|/|I|)^{\frac {\sigma_1} q}\le C\kappa_{t_0,x_0}^{-\frac{4m\sigma_1}{q}}.
\end{equation}
By the triangle inequality, we estimate the remaining terms on the right-hand side of \eqref{eq1.44} by
\begin{align}
                                \label{eq2.17}
&C\big(\kappa_{t_0,x_0}^{-(1-\frac{1}{q_0})}
+\kappa_{t_0,x_0}^{\frac{d+2m}{q_0}}\rho^{\frac{1}{q\varsigma}}\big)\sum_{|\alpha|\leq 2m}\lambda^{1-\frac{|\alpha|}{2m}}\|D^{\alpha}u_2\|_{L_{q,\omega}(\bR^{d+1}_+)}
+C\kappa_{t_0,x_0}^{\frac{d+2m}{q_0}}\|h\|_{L_{q,\omega}(\bR^{d+1}_+)}\nonumber\\
&\le C\big(\kappa_{t_0,x_0}^{-(1-\frac{1}{q_0})}
+\kappa_{t_0,x_0}^{\frac{d+2m}{q_0}}\rho^{\frac{1}{q\varsigma}}\big)
\bigg(\sum_{|\alpha|\leq 2m}\lambda^{1-\frac{|\alpha|}{2m}}\|D^{\alpha}u\|_{L_{q,\omega}(\bR^{d+1}_+)}\nonumber\\
&\quad +\sum_{|\alpha|\leq 2m}\lambda^{1-\frac{|\alpha|}{2m}}\|D^{\alpha}u_1\|_{L_{q,\omega}(\bR^{d+1}_+)}\bigg)
+C\kappa_{t_0,x_0}^{\frac{d+2m}{q_0}}\|h\|_{L_{q,\omega}(\bR^{d+1}_+)}.
\end{align}

Since $u=u_1+u_2$ and $h=f-(A-A(t_0,x_0))u_1$, from \eqref{eq1.44}, \eqref{eq1.47}, \eqref{eq2.49}, and \eqref{eq2.17} it follows that
\begin{align*}
&\|u_t\|_{L_{q,\omega}(\R^{d+1}_{+})}+\sum_{|\alpha|\leq 2m}\lambda^{1-\frac{|\alpha|}{2m}}\|D^{\alpha}u\|_{L_{q,\omega}(\R^{d+1}_{+})}\\
&\leq C\|\partial_t u_1\|_{L_{q,\omega}(\R^{d+1}_{+})}+C\kappa_{t_0,x_0}^{\frac{d+2m}{q_0}}
\sum_{|\alpha|\leq 2m }\lambda^{1-\frac{|\alpha|}{2m}}\|D^{\alpha}u_1\|_{L_{q,\omega}(\R^{d+1}_{+})}\\
&\quad+ C\kappa_{t_0,x_0}^{\frac{d+2m}{q_0}}\|f\|_{L_{q,\omega}(\R^{d+1}_{+})}
+C\big(\kappa_{t_0,x_0}^{-(1-\frac{1}{q_0})}
+\kappa_{t_0,x_0}^{\frac{d+2m}{q_0}}\rho^{\frac{1}{q_0\varsigma}}
+\kappa_{t_0,x_0}^{-\frac{4m\sigma_1}{q}}\big)\cdot\\
&\qquad\bigg(\|\partial_t u\|_{L_{q,\omega}(\R^{d+1}_{+})}+\sum_{|\alpha|\leq 2m,}\lambda^{1-\frac{|\alpha|}{2m}}
\|D^{\alpha}u\|_{L_{q,\omega}(\R^{d+1}_{+})}\bigg)
\end{align*}
which combined with \eqref{eq:prob4.55} yields
\begin{align*}
&\|u_t\|_{L_{q,\omega}(\R^{d+1}_{+})}+\sum_{|\alpha|\leq 2m}\lambda^{1-\frac{|\alpha|}{2m}}\|D^{\alpha}u\|_{L_{q,\omega}(\R^{d+1}_{+})}\\
&\leq C\kappa_{t_0,x_0}^{\frac{d+2m}{q_0}}
\sum_{j=1}^m\|g_j\|_{W^{k_j,2mk_j}_{q,\omega}(\partial \R^{d+1}_{+})}+ C\kappa_{t_0,x_0}^{\frac{d+2m}{q_0}}\|f\|_{L_{q,\omega}(\R^{d+1}_{+})}\\
&\quad+C_KC_{\varepsilon}\kappa_{t_0,x_0}^{\frac{d+2m}{q_0}}\|u\|_{L_{q,\omega}(\R^{d+1}_{+})}
+C\big(\kappa_{t_0,x_0}^{-(1-\frac{1}{q_0})}
+\kappa_{t_0,x_0}^{\frac{d+2m}{q_0}}\rho^{\frac{1}{q_0\varsigma}}
+\kappa_{t_0,x_0}^{-\frac{4m\sigma_1}{q}}+\varepsilon\kappa_{t_0,x_0}^{\frac{d+2m}{q_0}}\big)\cdot\\
&\qquad\bigg(\|\partial_t u\|_{L_{q,\omega}(\R^{d+1}_{+})}+\sum_{|\alpha|\leq 2m,}\lambda^{1-\frac{|\alpha|}{2m}}
\|D^{\alpha}u\|_{L_{q,\omega}(\R^{d+1}_{+})}\bigg)
\end{align*}
Now we take $\kappa_{t_0,x_0}$ sufficiently large, $\varepsilon$ sufficiently small, and then $\rho\le \rho_{t_0,x_0}$  sufficiently small such that
\[
C\big(\kappa_{t_0,x_0}^{-(1-\frac{1}{q_0})}
+\kappa_{t_0,x_0}^{\frac{d+2m}{q_0}}\rho^{\frac{1}{q_0\varsigma}}
+\kappa_{t_0,x_0}^{-\frac{4m\sigma_1}{q}}+\varepsilon\kappa_{t_0,x_0}^{\frac{d+2m}{q_0}}\big)\leq 1/2,
\]
and finally take $\lambda\geq \lambda_{t_0,x_0}:=\max\{\lambda_0,
2C_{K}C_{K,\varepsilon}\}$, we get \eqref{eq:VMOtimespaceLSgjtx} for $u$ with support in $\tilde Q_{\kappa_{t_0,x_0}^{-2}  r_{t_0,x_0}}(t_0,x_0)$ and
$$
C=C(\theta,d,m,K,q,b_{j\beta},t_0,x_0).
$$

Observe that $C$, $\rho_{t_0,x_0}$, and $\lambda_{t_0,x_0}$ all depend on $(t_0,x_0)$. However, since $\overline{\tilde{Q}_{R+1}}$ is compact, we can apply a partition of the unity argument and get a uniform constant $C$. This will be done in the next step.

{\em Step 3.} Observe first that since $\overline{\tilde Q_{R+1}}$ is compact and
$$
\overline{\tilde Q_{R+1}}\subset \bigcup_{(t_0,x_0)\in \overline{\tilde Q_{R+1}}} \tilde Q_{\kappa_{t_0,x_0}^{-2}  r_{t_0,x_0}/2}(t_0,x_0),
$$
there
exists a finite number $N\in\N$ of points $(t_{0,i},x_{0,i})\in \overline{\tilde Q_{R+1}}$, $i=1,\ldots,N$ such that
$$
\overline{\tilde Q_{R+1}}\subset \bigcup_{i=1}^{N}\tilde{Q}_{\kappa_{t_{0,i},x_{0,i}}^{-2}  r_{t_{0,i},x_{0,i}}/2}(t_{0,i},x_{0,i}).
$$
Take $\zeta_i\in C_0^{\infty}(\overline{\tilde{Q}_{\kappa_{t_{0,i},x_{0,i}}^{-2}  r_{t_{0,i},x_{0,i}}}(t_{0,i},x_{0,i})})$, $i=1,\ldots,N$, such that
$\zeta_i=1$ on $\tilde Q_{\kappa_{t_{0,i},x_{0,i}}^{-2}  r_{t_{0,i},x_{0,i}}/2}(t_{0,i},x_{0,i})$,
and $\zeta_{0}\in C_0^\infty(\overline{\R^{d+1}_{+}})$ such that
\begin{equation*}
\zeta_0(x)=
\begin{cases}
1 & x\in\overline{\R^{d+1}_{+}}\backslash\overline{\tilde{Q}_{R+1}}\\
0 & x\in\overline{\tilde{Q}_{R}}.
\end{cases}
\end{equation*}
Let $\overline{\zeta}=\sum_{i=0}^{N}\zeta_i^q\geq 1$ in $\overline{\R^{d+1}_{+}}$.
Define $\eta_i=\zeta_i(\overline{\zeta})^{-1/q}$. Then, $\sum_{i=0}^{N}\eta_i^q=1$ in $\R^{d+1}_{+}$.

Now we define
\[
u_{i}(t,x)=u(t,x)\eta_{i}(t,x).
\]
Observe that
\begin{equation}\label{prob:proofVMOtimespaceLStx}
\begin{cases}
\partial_t u_i+(A+\lambda)u_i=f_i\ \ {\rm in}\ \R^{d+1}_{+}
\\B_ju_{i}\big|_{x_1=0}=g_{j,i}\ \ {\rm on}\ \partial\R^{d+1}_{+},\  j=1,\ldots,m
\end{cases}
\end{equation}
where by Leibnitz's rule
\[
f_{i}=f\eta_{i}+u(\eta_{i})_{t}+\sum_{|\alpha|=2m}\sum_{|\gamma|\leq 2m-1}\binom{\alpha}{\gamma}a_{\alpha}(t,x)D^{\gamma}uD^{\alpha-\gamma}\eta_{i}\]
and
\[ g_{j,i}=g_j\eta_{i}+\sum_{1\leq |\beta|\leq m_j}\sum_{|\tau|\leq |\beta| -1}\binom{\beta}{\tau}b_{j\beta}(t,x)D^{\tau}uD^{\beta-\tau}\eta_{i}\big|_{x_1=0}.
\]
Now applying the result in \emph{Step 2} to \eqref{prob:proofVMOtimespaceLStx} we get for $i=1,\ldots,N$,
\begin{align*}
&\|(u_{i})_t\|_{L_{q,\omega}(\R^{d+1}_{+})}+\sum_{|\alpha|\leq 2m}\lambda^{1-\frac{|\alpha|}{2m}}\|D^{\alpha}u_{i}\|_{L_{q,\omega}(\R^{d+1}_{+})}\\
&\leq C_i\|f_{i}\|_{L_{q,\omega}(\R^{d+1}_{+})} +C_i\sum_{j=1}^m\|g_{j,i}\|_{W^{k_j,2mk_j}_{q,\omega}(\partial \R^{d+1}_{+})},
\end{align*}
with $C_i=C(\theta,d,m,K,q,[\omega]_{q},b_{j\beta},t_{0,i},x_{0,i})$, provided that $\lambda\ge \lambda_{t_{0,i},x_{0,i}}$ and $\rho\le \rho_{t_{0,i},x_{0,i}}$.
Applying the result in \emph{Step 1} to \eqref{prob:proofVMOtimespaceLStx} with $i=0$ we get a similar inequality, with $C_0=C(\theta,d,m,K,q,[\omega]_{q},\bar b_{j\beta})$.
Observe that by the triangle inequality,
\begin{align*}
&\|\eta_{i}u_t\|_{L_{q,\omega}(\R^{d+1}_{+})}+\sum_{|\alpha|\leq 2m}\lambda^{1-\frac{|\alpha|}{2m}}\|\eta_{i}D^{\alpha}u\|_{L_{q,\omega}(\R^{d+1}_{+})}\\
&\leq \|(u_{i})_t\|_{L_{q,\omega}(\R^{d+1}_{+})}+\|u(\eta_{i})_t\|_{L_{q,\omega}(\R^{d+1}_{+})}
+\sum_{|\alpha|\leq 2m}\lambda^{1-\frac{|\alpha|}{2m}}\|D^{\alpha}u_{i}\|_{L_{q,\omega}(\R^{d+1}_{+})}\\
&\quad +\sum_{|\alpha|= 2m}\sum_{|\gamma|\leq 2m-1}\binom{\alpha}{\gamma}\lambda^{1-\frac{|\alpha|}{2m}}\|D^{\gamma}u D^{\alpha-\gamma}\eta_{i}\|_{L_{q,\omega}(\R^{d+1}_{+})},
\end{align*}

\begin{multline*}
\|f_{i}\|_{L_{q,\omega}(\R^{d+1}_{+})}\leq\|f\eta_{i}\|_{L_{q,\omega}(\R^{d+1}_{+})}+ \|u(\eta_{i})_t\|_{L_{q,\omega}(\R^{d+1}_{+})}\\
+C_{K}\sum_{|\alpha|= 2m}\sum_{|\gamma|\leq 2m-1}\binom{\alpha}{\gamma}\|D^{\gamma}u D^{\alpha-\gamma}\eta_{i}\|_{L_{q,\omega}(\R^{d+1}_{+})},
\end{multline*}
and
\begin{multline*}
\|g_{j,i}\|_{W^{k_j,2mk_j}_{q,\omega}(\partial \R^{d+1}_{+})}
\leq \|g_j\eta_{i}\|_{W^{k_j,2mk_j}_{q,\omega}(\partial \R^{d+1}_{+})}\\+C_{K}\sum_{1\leq |\beta|\leq m_j}\sum_{|\tau|\leq |\beta| -1}\binom{\beta}{\tau}\|D^{\tau}uD^{\beta-\tau}\eta_{i}\|_{W^{\frac{2m-m_j}{2m},2m-m_j}_{q,\omega}(\R^{d+1}_{+})},
\end{multline*}
where we used the boundedness of the coefficients $a_\alpha$ and $b_{j\beta}$.
After taking the $q$-th power, and summing in $i=0,1,\ldots,N$, letting
$$
C=C_0+\sup_{i=1,\ldots,N}C(\theta,d,m,K,q,[\omega]_q,b_{j\beta},t_{0,i},x_{0,i})
$$
and taking the $q$-th root, we get
\begin{align*}
&\|u_t\|_{L_{q,\omega}(\R^{d+1}_{+})}+\sum_{|\alpha|\leq 2m}\lambda^{1-\frac{|\alpha|}{2m}}\|D^{\alpha}u\|_{L_{q,\omega}(\R^{d+1}_{+})}\\
&\leq C\|f\|_{L_{q,\omega}(\R^{d+1}_{+})}+ C\sum_{j=1}^m\|g_j\|_{W^{k_j,2mk_j}_{q,\omega}(\partial\R^{d+1}_{+})}
+C\|u\|_{L_{q,\omega}(\R^{d+1}_{+})}\\
&\quad+C_{K}\sum_{|\alpha|= 2m}\sum_{|\gamma|\leq 2m-1}\binom{\alpha}{\gamma}\lambda^{1-\frac{|\alpha|}{2m}}\|D^{\gamma}u D^{\alpha-\gamma}\eta_{i}\|_{L_{q,\omega}(\R^{d+1}_{+})}\\
&\quad +C_{K}\sum_{1\leq |\beta|\leq m_j}\sum_{|\tau|\leq |\beta| -1}\binom{\beta}{\tau}\lambda^{1-\frac{|\tau|}{m_j}}
\|D^{\tau}u\|_{W^{\frac{2m-m_j}{2m},2m-m_j}_{q,\omega}(\R^{d+1}_{+})}
\end{align*}
with $C$ uniform in $t_{0,i},x_{0,i}$, provided that
$$
\lambda\ge \lambda':= \max\{\bar\lambda,\lambda_{t_{0,i},x_{0,i}}:i=1,\ldots,N\},
\quad \rho\le \rho':=\min\{\bar\rho,\rho_{t_{0,i},x_{0,i}}:i=1,\ldots,N\}.
$$

This, combined with interpolation estimates and taking $\varepsilon$ small and $\lambda$ large, gives \eqref{eq:VMOtimespaceLSgjtx} with $p=q$ and $\omega\in A_q(\bR)$ such that $[\omega]_{q}\leq \Lambda_0$, i.e.,
\begin{multline}\label{eq:VMOtimespaceLSgjtxp}
\|u_t\|_{L_{q,\omega}(\R^{d+1}_+)}+\sum_{|\alpha|\le 2m}\lambda^{1-\frac{|\alpha|}{2m}}\|D^{\alpha}u\|_{L_{q,\omega}(\R^{d+1}_+)}\\
\le C\|f\|_{L_{q,\omega}(\R^{d+1}_+)} + C\sum_{j=1}^m\|g_j\|_{W^{k_j,2mk_j}_{q,\omega}(\partial\R^{d+1}_{+})},
\end{multline}
where $k_j=1-m_j/(2m)-1/(2mq)$ and $C=C(\theta,m,d,K,q,\Lambda_0,b_{j\beta})>0$.

\emph{Step 4.} We now extrapolate the estimate from the previous step to $p\neq q$. By \eqref{eq:VMOtimespaceLSgjtxp} and Definition \ref{def:TLspace}, we have that for all $\omega\in A_q(\bR)$ such that $[\omega]_{A_{q}}\leq \Lambda_0$ there exist constants $\lambda', \rho', C>0$ depending on $\Lambda_0$ such that for any $\lambda\ge \lambda'$ and $\rho\le \rho'$,
\begin{multline}
\sum_{|\alpha|\le 2m}\lambda^{1-\frac{|\alpha|}{2m}}\|U_{\alpha}\|_{L_q(\R,\omega)}\\
 \leq C  \|F\|_{ L_{q}(\R,\omega)} + C\sum_{j=1}^m\|G_{j,1}\|_{L_{q}(\R,\omega)}
 +C\sum_{j=1}^m\|G_{j,2}\|_{L_{q}(\R,\omega)},
 \nonumber
\end{multline}
where
$$
U_{\alpha} = \|D^{\alpha}u\|_{L_{q}(\R^{d}_{+})},\quad
F = \|f\|_{L_{q}(\R^{d}_{+})},
$$
$$
G_{j,1}=\|2^{k k_j}\mathcal{F}^{-1}(\widehat{\varphi}_{k}\hat{g}_j)_{k\geq 0}\|_{\ell_{q}(L_{q}(\R^{d-1}))},\quad
G_{j,2}=\|g_{j}\|_{\B_{q,q}^{2mk_j}(\R^{d-1})}.
$$
Since the above estimate holds for all of the $A_{q}$ weights with uniformly bounded $A_q$-constant, $\rho'$ and $\lambda'$ can be chosen uniformly. Therefore, by the extrapolation result Theorem \ref{thm:extensionRubio} it follows that for all $\omega\in A_p$,  there exist a constant $C'$ depending on $[\omega]_{p}$ such that for all $\lambda\geq \lambda'$ and $\rho\le \rho'$,
\begin{multline}
\sum_{|\alpha|\le 2m}\lambda^{1-\frac{|\alpha|}{2m}}\|U_{\alpha}\|_{L_p(\R,\omega)}\\
\leq C'  \|F\|_{ L_{p}(\R,\omega)} + C'\sum_{j=1}^m\|G_{j,1}\|_{ L_{p}(\R,\omega)}+C'\sum_{j=1}^m\|G_{j,2}\|_{L_{p}(\R,\omega)}.
\nonumber
\end{multline}
This yields
\begin{multline}
\sum_{|\alpha|\le 2m}\lambda^{1-\frac{|\alpha|}{2m}}\|D^{\alpha}u\|_{L_p(\R,\omega;L_{q}(\R^{d}_{+}))}
\leq C \|f\|_{L_{p}(\R, \omega;L_{q}(\R^{d}_{+}))}\\
+ C\sum_{j=1}^m\|g_j\|_{F^{k_j}_{p,q}(\R,\omega;L_{q}(\R^{d-1}))\cap L_{p}(\R,\omega;\B_{q,q}^{2mk_j}(\R^{d-1}))},
\nonumber
\end{multline}
with $C=C(\theta,m,d,K,p,q,[\omega]_{p},b_{j\beta})$. As $u_t = f-(\lambda+A)u$, the estimate \eqref{eq:VMOtimespaceLSgjtx} directly follows.
\end{proof}

\section{Estimates on domains}\label{sec:domain}

In this section let $\Omega\subset\R^{d}$ be a bounded smooth domain of the class $C^{2m-1,1}$. In the following, we generalize Theorem \ref{thm:VMOproblemLStx} to this setting, i.e., we will prove $L_p(L_q)$--estimates for the parabolic problem
\begin{equation*}
\begin{cases}
u_t + (A(t,x)+\lambda)u=f & {\rm in}\ \R\times\Omega\\
{\rm tr}_{\partial\Omega}B_{j}(t,x)u=g_j & {\rm on}\ \R\times\partial\Omega.
\end{cases}
\end{equation*}

We impose slightly different assumptions in order to adapt to the bounded domain case.
Define the operators $A$ and $B_j$ as in Section \ref{sec:assumptions}, where $a_\alpha:\R\times\Omega\rightarrow\C$ and $b_{j\beta}:\R\times\Omega\rightarrow\C$.
Let $R_0\in (0,1]$, $K>0$, and $\theta\in (0,\pi/2)$ be fixed constants. The parameter--ellipticity condition is formulated as follows.
\let\ALTERWERTA\theenumi
\let\ALTERWERTB\labelenumi
\def\theenumi{(E)$_\theta$}
\def\labelenumi{\textbf{(E)}$_\theta$}
\begin{enumerate}
	\item\label{as:condelldomain}
	For all $t\in\R$ and $x\in\Omega$, it holds that
	\begin{equation*}
	\sigma(A_{\sharp}(t,x,\xi))\subset\Sigma_{\theta},\quad  \forall\ \xi\in\R^{n},\ |\xi|=1,
	\end{equation*}
	for the spectrum of the operator $A_{\sharp}(t,x,\xi)$.
\end{enumerate}
\let\theenumi\ALTERWERTA
\let\labelenumi\ALTERWERTB

Before stating the Lopatinskii--Shapiro condition, we need to introduce some notation.
For each $x_0\in \partial\Omega$, there is a local coordinate system such that $x_0$ is the origin and $e_1$ is the normal direction at $x_0$ \footnote{Here $[e_j]_{j=1}^{d}$ denotes the standard basis of $\R^d$.}. For $t_0\in \R$ and $x\in B_{2R_0}(x_0)\cap\Omega$, consider the operators $A^H(t_0,x,D), B_{j}^{H}(t_0,x,D)$ as in \eqref{eq:opH} and write the boundary problem $(A^H(t_0,x,D), B_{j}^{H}(t_0,x,D))$ in the local coordinates corresponding to $x_0$. We then assume that the (LS)$_\theta$--condition holds for any $x\in B_{2R_0}(x_0)\cap\Omega$ with respect to this coordinate system, which can be stated as follows.

\let\ALTERWERTA\theenumi
\let\ALTERWERTB\labelenumi
\def\theenumi{(LS)$_\theta$}
\def\labelenumi{\textbf{(LS)$_\theta$}}
\begin{enumerate}
	\item\label{as:LSconddomain}
	For each $(h_{1},\ldots,h_{m})^{T}\in\R^{d-1}$, $\xi\in\R^{m}$, $\displaystyle \lambda\in\overline{\Sigma}_{\pi-\theta}$, and $(t_0,x)\in \overline{\R\times\Omega}$ such that $x\in B_{2R_0}(x_0)\cap\Omega$ for some $x_0\in\partial\Omega$ and $|\xi|+|\lambda|\neq 0$, the ODE problem in $\R_+$
	\begin{equation*}
	\begin{cases}
	\lambda v+ A^H(t_0,x,\xi,D_{x_1})v=0,\quad  x_1 >0,\\
	{\rm tr}_{\partial\Omega}B^H_{j}(t_0,x,\xi,D_{x_1})v=h_{j},\quad  j=1,\ldots,m,
	\end{cases}
	\end{equation*}
	admits a unique solution $v\in C^{\infty}(\R_+)$ such that $\lim_{x\rightarrow \infty}v(x)=0$.
\end{enumerate}
\let\theenumi\ALTERWERTA
\let\labelenumi\ALTERWERTB
We also modify Assumption \ref{ass:VMO}.
\begin{assumption}[$\rho$] \label{ass:VMOdomain}
	For $|\alpha|=2m$, there exist a constant $R_0\in (0,1]$ such that
	$$
	\sup_{(t,x)\in\R^{d+1}}\sup_{r\leq R_0}\avint_{Q^{\Omega}_{r}(t,x)}|a_{\alpha}(s,y)-(a_{\alpha})_{Q^{\Omega}_{r}(t,x)}|\leq \rho,
	$$
where $Q^{\Omega}_{r}(t,x):=((t-r^{2m},t)\times B_{r}(x))\cap(\R\times\Omega)$.
\end{assumption}
We finally impose the following assumptions on the coefficients of $A$ and $B_j$.
\let\ALTERWERTA\theenumi
\let\ALTERWERTB\labelenumi
\def\theenumi{(A)}
\def\labelenumi{\textbf{(A)}}
\begin{enumerate}
	\item\label{as:operatorALSdomain} For the multi-index $\alpha$, the coefficients $a_{\alpha}$ are functions
	$\R\times\Omega\rightarrow\C$, $\|a_\alpha\|_{L_\infty}\le K$,
	and satisfy Assumption \ref{ass:VMOdomain} ($\rho$) with a parameter $\rho\in (0,1)$ to be determined later.
	Moreover, $A$ satisfies condition $(E)_{\theta}$.
\end{enumerate}
\let\theenumi\ALTERWERTA
\let\labelenumi\ALTERWERTB

\let\ALTERWERTA\theenumi
\let\ALTERWERTB\labelenumi
\def\theenumi{(B)}
\def\labelenumi{\textbf{(B)}}
\begin{enumerate}
	\item\label{as:operatorBLSdomain} The coefficients $b_{j\beta}:\R\times\Omega\rightarrow\C $ satisfy
	$$
	b_{j\beta}\in C^{\frac{2m-m_{j}}{2m},2m-m_{j}}(\R\times\Omega),\quad
	\|b_{j\beta}\|_{C^{\frac{2m-m_{j}}{2m},2m-m_{j}}(\R\times\Omega)}\leq K,
	$$
	and
	$$
	\lim_{|t|+|x|\rightarrow \infty}b_{j\beta}(t,x)=\overline{b}_{j\beta}.
	$$
	The \ref{as:LSconddomain}-condition is satisfied by $(A,\overline{B}_{j})$ for any $A\in (E)_{\theta}$, where $\overline{B}_{j}$, $j=1,\ldots,m$, are the boundary operators with coefficients $\overline{b}_{j\beta}$.
\end{enumerate}
\begin{theorem}\label{thm:VMOdomainLStx}
Let $T\in(-\infty,\infty]$, $p,q\in(1,\infty)$, and $\omega\in A_p(\R)$. Let $\Omega$ be a $C^{2m-1,1}$-domain with the $C^{2m-1,1}$-norm bounded by $K$. There exists
$$
\rho=\rho(\theta,m,d,K,p,q,b_{j\beta},[\omega]_{p})\in (0,1)
$$
such that under the assumptions \ref{as:operatorALS}, \ref{as:operatorBLS}, and \ref{as:LSconddomain} the following holds.
There exists $\lambda_0=\lambda_0(\theta,m,d,K,p,q,R_0,b_{j\beta})\geq 1$ such that for every $\lambda\geq\lambda_0$, for
$$
u\in W^{1}_{p}((-\infty,T),\omega;L_{q}(\Omega))\cap L_{p}((-\infty,T),\omega;W^{2m}_{q}(\Omega))
$$
satisfying the problem
\begin{equation}\label{prob:thmVMOdomLStx}
\begin{cases}
u_t + (A(t,x)+\lambda)u=f & {\rm in}\ \R\times\Omega\\
{\rm tr}_{\partial\Omega}B_{j}(t,x)u=g_j & {\rm on}\ \R\times\partial\Omega,
\end{cases}
\end{equation}
where $f\in L_p((-\infty,T),\omega;L_{q}(\Omega))$ and with $k_j=1-m_j/(2m)-1/(2mq)$, $g_j\in W^{k_j,2mk_j}_{p,q}((-\infty,T),\omega\times \partial\Omega)$, it holds
\begin{multline*}
\|u_t\|_{L_{p}((-\infty,T),\omega;L_{q}(\Omega))}+\sum_{|\alpha|\leq 2m}\lambda^{1-\frac{|\alpha|}{2m}}\|D^{\alpha}u\|_{L_{p}((-\infty,T),\omega;L_{q}(\Omega))}\\
\leq C\|f\|_{L_{p}((-\infty,T),\omega;L_{q}(\Omega))} + C\sum_{j=1}^m\|g_j\|_{W^{k_j,2mk_j}_{p,q,\omega}
((-\infty,T)\times \partial\Omega)},
\end{multline*}
with a constant $C=C(\theta,m,d,K,p,q,[\omega]_{p},b_{j\beta})>0$.
\end{theorem}

The proof of Theorem \ref{thm:VMOdomainLStx} follows from the technique of flattening the boundary. For this, we take an \textit{admissible} $C^{2m-1,1}$-coordinate transformation as in \cite[Section 8]{DHP}, so that the parameter-ellipticity condition and the \ref{as:LSconddomain}-condition are preserved. For further details on admissible coordinate transformations, we refer the reader to \cite[Section 2]{WlokaBook}.

Let $x_0$ be in a neighborhood of $\partial\Omega$ of width $2 R_0$  and choose coordinates corresponding to $x_0$. By definition of a $C^{2m-1,1}$-boundary (see \cite[Definition 2.4 and Theorem 2.6]{WlokaBook}), there exists an open neighborhood $B=B_1\times B_2\subset\R^d$ containing $x_0$ with $B_1\subset\R^{d-1}$ and $B_2\subset\R$ open and a function $\varphi\in C^{2m-1,1}(\overline{B_1})$ satisfying
$$
\partial\Omega\cap B=\{x=(x_1,x')\in B:\ x_1=\varphi(x')\}
$$
and
$$
\Omega\cap B=\{x\in B:\ x_1 >\varphi(x')\}.
$$
Setting
\[
\Phi(x):=\binom{x'}{x_1-\varphi(x')},\ \ x\in B,
\]
$\Phi:B\rightarrow \Phi(B), x\mapsto y$, and proceeding as in \cite[Section 8]{DHP}, the differential operators $A$ and $B_j$, $j=1,\ldots,m$, are transformed into the operators
\[
A^{\Phi}=\sum_{|\alpha|\leq 2m}a_{\alpha}^{\Phi}(t,y)D^{\alpha},\ \
B_j^{\Phi}=\sum_{|\beta|\leq m_j} b_{j\beta}^{\Phi}(t,y)D^{\beta}
\]
and act on functions defined on $\Phi(B)\cap\R^{d}_{+}$.
By \cite[Theorem 10.3]{WlokaBook}, the principal symbol of $A^{\Phi}$ is given by
\[
A^{\Phi}_{\sharp}(y,\xi)=A_{\sharp}(\Phi^{-1}(y),[D\Phi(\Phi^{-1}(y))]^{T}\xi),\ \ y\in\Phi(B)\cap\overline{\R^{d}_{+}},\ \ \xi\in\R^{d}.
\]
As $D\Phi$ is an isomorphism of $\R^d$ for all $x\in B$, this implies that parameter-ellipticity of $A^{\Phi}$ and, in particular, the condition $(E)_{\theta}$ are preserved under coordinate transformations.
Moreover, if $\Phi$ is admissible at $x_0$ then the transformed boundary problem $(A^{\Phi},B_j^{\Phi})$ satisfies the \ref{as:LSconddomain}-condition on $\R^{d}_{+}$ at the point $\Phi(x_0)$. This can be seen in the same way as \cite[Theorem 11.3]{WlokaBook}. Finally, it is easily seen that the leading coefficients of the new operator in the $y$-coordinates also satisfy Assumption \ref{ass:VMOdomain} with a possibly different $\rho$.
We remark that $u$ satisfies
\begin{equation*}
\begin{cases}
u_t + (A(t,x)+\lambda)u=f & {\rm in}\ \R\times(B\cap \Omega)\\
{\rm tr}_{\partial\Omega}B_{j}(t,x)u=g_j & {\rm on}\ \R\times (B\cap \partial\Omega)
\end{cases}
\end{equation*}
if and only if the transformed function $u^{\Phi}$ satisfies for $j=1,\dots,m$,
\begin{equation}\label{prob:VMOLStxtransformed}
\begin{cases}
\partial_tu^{\Phi}+ (\lambda+A^{\Phi})u^{\Phi}=f^{\Phi} & {\rm in}\ \R\times\Phi(B)\cap\R^{d}_{+}\\
{\rm tr}_{\R^{d}_{+}}B_j^{\Phi}(t,x)u^{\Phi}=g_{j}^{\Phi} & {\rm on}\ \R\times
\Phi(B)\cap\R^{d-1}.
\end{cases}
\end{equation}

\begin{proof}[Proof of Theorem \ref{thm:VMOdomainLStx}]
As in the proof of Theorem \ref{thm:VMOproblemLStx} it suffices to consider $T=\infty$.

Given a $C^{2m-1,1}$-boundary $\partial\Omega$, by \cite[Theorem 2.11]{WlokaBook} each point $x_0$ in a neighborhood of $\partial\Omega$ of width $2R_0$, possesses an open neighborhood $B(x_0)$ and an admissible $C^{2m-1,1}$-coordinate transformation $\Phi=\Phi_{x_0}:B(x_0)\rightarrow\R^{d}$ with the above properties. Denote now by $\Psi_{x_0}$ the push-forward operator corresponding to $\Phi_{x_0}$ and define the transformed differential operators $A^{\Phi_{x_0}}$ and $B_j^{\Phi_{x_0}}$, $j=1,\ldots,m$, acting on functions defined on $\Phi_{x_0}(B(x_0))\cap\R^{d}_{+}$.
The proof of Theorem \ref{thm:VMOdomainLStx} when $p=q$ then follows from Theorem \ref{thm:VMOproblemLStx} and a partition of the unity argument as in for instance \cite[Theorem 6]{DK11}, so we omit the details. The general case is then derived from the case when $p=q$ and Theorem \ref{thm:extensionRubio} as in the proof of Theorem \ref{thm:VMOproblemLStx}.
\end{proof}
In the same way as Theorem \ref{thm:VMOellipticLStx} is obtained from Theorem \ref{thm:VMOproblemLStx}, one can state and show the following elliptic version of Theorem \ref{thm:VMOdomainLStx}. Note that because $\Omega$ is assumed to be bounded, the limit behavior of $b_{j\beta}$ in the assumption (B) is unnecessary.
\begin{theorem}\label{thm:VMOellipticdomainLStx}
Let $q\in(1,\infty)$. Let $\Omega$ be a $C^{2m-1,1}$-domain with the $C^{2m-1,1}$-norm bounded by $K$. There exists
$
\rho=\rho(\theta,m,d,K,q,b_{j\beta})\in (0,1)
$ such that under the assumptions \ref{as:operatorALS}, \ref{as:operatorBLS}, and \ref{as:LSconddomain}, the following holds.
There exists $\lambda_0=\lambda_0(\theta,m,d,K,q,R_0,b_{j\beta})\geq 1$ such that for every $\lambda\geq\lambda_0$, for $u\in W^{2m}_{q}(\Omega)$ satisfying the problem
\begin{equation*}
\begin{cases}
(A(x)+\lambda)u=f & {\rm in}\ \Omega\\
{\rm tr}_{\partial\Omega}B_{j}(x)u=g_j & {\rm on}\ \partial\Omega,
\end{cases}
\end{equation*}
where $f\in L_{q}(\Omega)$, and with $k_j=(2m-m_j-1/q)/(2m)$, $g_j\in W_{q}^{2mk_j}(\partial\Omega)$, it holds
\begin{equation*}
\sum_{|\alpha|\leq 2m}\lambda^{1-\frac{|\alpha|}{2m}}\|D^{\alpha}u\|_{L_{q}(\Omega)}\leq C\|f\|_{L_{q}(\Omega)} + C\sum_{j=1}^m\|g_j\|_{W_{q}^{2mk_j}(\partial\Omega)},
\end{equation*}
with a constant $C=C(\theta,m,d,K,q,b_{j\beta})>0$.
\end{theorem}
\begin{remark}
\begin{enumerate}
\item[(i)] The a priori estimates in Theorems \ref{thm:VMOproblemLStx}, \ref{thm:VMOdomainLStx} and the corresponding elliptic results can be used to derive the existence of solutions to the corresponding equations. This can be shown in the same way as in \cite [Section 6]{DG17}.
\item[(ii)] For notational simplicity, we consider the scalar case only. However, with the same proofs Theorems \ref{thm:VMOproblemLStx} and \ref{thm:VMOdomainLStx} and the corresponding elliptic results also hold if one considers systems of operators.
\end{enumerate}
\end{remark}

\def\cprime{$'$}

\end{document}